\documentclass[11pt]{amsart}

\usepackage[T1]{fontenc}
\usepackage[utf8]{inputenc}
\usepackage[
  backend=biber,
  style=numeric,
  sorting=none,
  giveninits=false,
  doi=true,
  isbn=true,
  url=true,
  eprint=true
]{biblatex}
\usepackage{lmodern}
\usepackage{microtype}
\usepackage{tabularx}
\usepackage{array}
\usepackage{amsmath,amssymb,amsfonts,amsthm}
\usepackage{mathtools}
\usepackage{enumitem}
\usepackage{tikz-cd}
\usepackage{geometry}
\usepackage{float}

\usepackage[colorlinks=true,linkcolor=blue,citecolor=blue,urlcolor=blue]{hyperref}
\usepackage[nameinlink,noabbrev]{cleveref}

\hypersetup{
  pdftitle={Cycle Relations and Global Gluing in Multi-Node Conifold Degenerations},
  pdfauthor={Abdul Rahman}
}

\DeclareMathOperator{\Cone}{Cone}

\DeclareMathOperator{\id}{id}

\numberwithin{equation}{section}

\theoremstyle{plain}
\newtheorem{theorem}{Theorem}[section]
\newtheorem{proposition}[theorem]{Proposition}
\newtheorem{lemma}[theorem]{Lemma}
\newtheorem{corollary}[theorem]{Corollary}

\newtheorem{hypothesis}[theorem]{Hypothesis}
\theoremstyle{definition}
\newtheorem{definition}[theorem]{Definition}

\theoremstyle{remark}
\newtheorem{remark}{Remark}[section]

\makeatletter



\makeatother

\newcommand{\Q}{\mathbb{Q}}

\newcommand{\C}{\mathbb{C}}

\newcommand{\Perv}{\mathrm{Perv}}

\newcommand{\Ext}{\mathrm{Ext}}

\newcommand{\can}{\mathrm{can}}
\newcommand{\var}{\mathrm{var}}
\newcommand{\rat}{\mathrm{rat}}

\newcommand{\MHM}{\mathrm{MHM}}
\newcommand{\PP}{\mathrm{PP}}

\newcommand{\rank}{\mathrm{rank}}
\newcommand{\Span}{\mathrm{Span}}

\title[Cycle Relations and Global Gluing]{Cycle Relations and Global Gluing in \\ Multi-Node Conifold Degenerations}

\author{Abdul Rahman}
\thanks{Email: arahman@alum.howard.edu}

\begin{document}
\begin{abstract}
We study projective one-parameter conifold degenerations whose central fiber has finitely many ordinary double points. Existing finite-node theory isolates one rank-one local sector per node on the perverse-sheaf, mixed-Hodge-module, and categorical sides, but does not determine which global extension classes are actually realized by geometry. We show that when the nodes are linked by common cycle geometry or homological relations, the corrected extension is not free nodewise data, but is forced into a smaller relation-controlled subspace. To formalize this, we introduce a cycle-node incidence datum and the associated geometrically realized subspace of the ambient nodewise extension space. Under geometrically admissible and block-adapted incidence hypotheses, we prove that the corrected perverse extension factors through this subspace, with incidence compatibility derived from propagation of local variation data along admissible cycle components, and we show that the same relation law lifts compatibly to the mixed-Hodge-module setting. We then compare this relation law with the resolution and smoothing sides and, in the block-separated cycle family, obtain
$R_{\mathrm{res}}=R_{\mathrm{sm}}=R_{\mathrm{ext}}=R_{\mathrm{blk}}$.
\end{abstract}
\maketitle

\section{Introduction}

\subsection{Conifold degenerations and the local-to-global problem}

A basic lesson of the compatibility analysis in \cite{HubschRahman2002} is that nodes in a conifold-type degeneration need not behave independently: when several nodes lie on a common distinguished cycle, the corresponding local geometric data may be forced to coincide globally. The present paper studies the sheaf-theoretic form of this phenomenon. More precisely, for a projective one-parameter degeneration 
$\pi:X\to \Delta$ whose central fiber \(X_0\) has finitely many ordinary double points $\Sigma=\{p_1,\dots,p_r\}\subset X_0$,
we ask whether the corrected finite-node extension is genuinely free nodewise data, or whether the ambient geometry forces it into a smaller relation-controlled subspace.

Such finite-node conifold degenerations form one of the basic geometric models of topology change in complex dimension three. Their local and global geometry is governed by the interaction of vanishing cycles, Picard--Lefschetz monodromy, small resolutions, and smoothing theory; in particular, relations among exceptional curves and vanishing spheres have long been known to control key aspects of conifold transitions \cite{Friedman1986,LeeLinWang2021,FriedmanLaza2024,HubschBestiary2nd2024}. Locally, each ordinary double point contributes a rank-one vanishing cycle, so the finite-node singular sector begins as a formal sum of \(r\) rank-one local contributions.

The sheaf-theoretic and Hodge-theoretic structure of this finite-node situation has been developed in earlier work. On the perverse side, one associates to the degeneration the corrected object
\[
\mathcal P:=\Cone(\var_F)[-1], \qquad F:=\Q_X[3],
\]
constructed from the variation morphism between vanishing and nearby cycles. In the finite-node ordinary-double-point setting, this object fits into an exact sequence
\[
0\to IC_{X_0}\to \mathcal P\to \bigoplus_{k=1}^r i_{k*}\Q_{\{p_k\}}\to 0,
\]
so that the singular quotient is the direct sum of the point-supported node sectors \cite{RahmanSchoberPaper,RahmanPerverseNearbyCycles,RahmanMixedHodgeModules,RahmanMultiNodeSchoberPaper}. On the mixed-Hodge-module side, the same corrected object admits an internal refinement in Saito's category,
\[
0\to IC^H_{X_0}\to \mathcal P^H\to \bigoplus_{k=1}^r i_{k*}\Q^H_{\{p_k\}}(-1)\to 0,
\]
so that the local rank-one singular sectors persist as Tate-twisted point-supported mixed Hodge modules \cite{SaitoMHM,SaitoYoungGuide,SaitoDuality,RahmanPerverseNearbyCycles,RahmanMixedHodgeModules}. On the categorical side, the same finite-node geometry admits a schober-type organization with one localized sector per node and a quiver shadow recovering the corrected extension after decategorification \cite{RahmanMultiNodeSchoberPaper}. Thus the local ordinary-double-point block, its finite direct-sum assembly, and its mixed-Hodge and categorical lifts are already in place.

What is not yet understood is the global extension problem. The finite-node corrected object has nodewise singular quotient \(\bigoplus_{k=1}^r i_{k*}\Q_{\{p_k\}}\), but this quotient alone does not determine which global extension classes are geometrically realizable. More precisely, the direct-sum form of the singular quotient explains why the \emph{ambient} extension problem admits one formal nodewise direction per singular point; it does \emph{not} imply that the distinguished corrected extension class coming from a single global degeneration varies independently at each node. The issue is therefore not local existence, but global compatibility: which linear combinations of the local node sectors actually arise from the canonical corrected object once the nodes are linked by the ambient geometry?

The guiding observation is that the nodes of a finite conifold degeneration need not be independent. They may arise as zeros of a common holomorphic condition, lie on common global cycle components, or become linked by the same homological data after passing to a resolution. In such situations, the node set carries additional global structure, and one expects this structure to impose relations on the corrected-extension coefficients. The purpose of the present paper is to make this precise by introducing a cycle-node incidence datum \((\mathcal C,\iota_{\mathcal C})\) and using it to define a geometrically realized subspace inside the free nodewise extension space. The basic claim is that the corrected perverse extension, its mixed-Hodge-module refinement, and its quiver shadow all factor through this smaller relation-controlled space rather than through the full free nodewise one.

The main benefit of this reformulation is that it isolates the first genuinely global invariant of the finite-node corrected-extension package. It shows that the corrected extension is not merely a formal sum of local rank-one pieces, but a global object constrained by the homological geometry of the node configuration. In particular, it provides a theorem-level bridge between classical conifold-transition data on the smoothing and resolution sides and the corrected perverse/MHM formalism on the sheaf-theoretic side. This is the structural input needed for later transport, wall crossing, and quiver/BPS developments: those later theories should act not on the formally free nodewise space, but on the relation-controlled geometric subspace identified here.

\subsection{Review of the literature}

The literature surrounding conifold degenerations already contains the local ordinary-double-point package, the classical smoothing--resolution comparison, and several sheaf-theoretic, Hodge-theoretic, and categorical frameworks in which degeneration data can be organized. What appears to be missing is a theorem-level account of the specifically global question studied here: when a projective finite-node degeneration has nodes linked by common cycle geometry or homological relations, does the corresponding corrected extension remain free nodewise data, or is it forced into a smaller relation-controlled subspace? The point is not the existence of local singular sectors, which is well understood, but the global constraint problem: how the ambient geometry of the node configuration governs which extension classes are actually realized.

The sheaf-theoretic background for the present paper begins with the theory of perverse sheaves and nearby/vanishing cycles developed by Beilinson, Bernstein, and Deligne \cite{BBD}, together with the gluing descriptions of perverse sheaves due to MacPherson and Vilonen \cite{MacPhersonVilonen} and the general sheaf-theoretic formalism of Kashiwara and Schapira \cite{KS}. A particularly relevant modern precedent is the work of Banagl, Budur, and Maxim \cite{BanaglBudurMaxim2012}, who construct nearby-cycle-based correction objects for hypersurfaces with isolated singularities, prove perversity, and show that these objects underlie mixed Hodge modules. Their work shows that correction objects arising from isolated singularities can admit internal mixed-Hodge-theoretic refinements. However, their setting and correction object differ from the corrected finite-node conifold extension studied here, and they do not address the classification of global extension data in the presence of homological relations among several nodes.

A second foundational strand comes from the degeneration and limiting-mixed-Hodge-structure literature. Deligne, Schmid, and Steenbrink established the framework through which local vanishing data contributes to global degenerating cohomology and limiting Hodge structures \cite{DeligneDegeneration,Schmid,SteenbrinkLimits}. These works matter here not because they already contain the corrected finite-node extension studied in this paper, but because they make precise the local-to-global principle underlying it: nearby and vanishing cycles are not merely local invariants of singular points, but part of a global degeneration package constrained by monodromy, specialization, and the organization of the singular locus inside the total space. They therefore provide the broader Hodge-theoretic backdrop for asking whether local nodewise correction data can really vary freely, or whether the ambient degeneration geometry forces it into a smaller relation-controlled space.

On the geometric side, conifold transitions have long been understood through the interplay of smoothing and resolution. Friedman's classical work on smoothing threefold double points \cite{Friedman1986} shows that relations among exceptional curves govern the deformation problem, while Friedman and Laza place related smoothing questions into a broader modern framework for mildly singular Calabi--Yau varieties \cite{FriedmanLaza2024}. From a complementary perspective, Lee, Lin, and Wang develop an $A+B$ theory for projective conifold transitions, showing that exceptional-curve data on the resolution side and vanishing-cycle data on the smoothing side together control substantial geometric information across the transition \cite{LeeLinWang2021}. Chi Li's work on Clemens manifolds further shows that nontrivial Hodge-theoretic structure survives on the smoothing side \cite{ChiLi2022}. These works strongly support the view that exceptional curves and vanishing cycles are the correct bridge objects for conifold geometry. What they do not provide is a theorem-level description of how such relation data constrains corrected perverse or mixed-Hodge-module extension classes.

The categorical literature supplies a further line of motivation. Perverse schobers, introduced as categorical analogues of perverse sheaves, provide a natural setting in which local monodromy and wall-crossing phenomena may be categorified. Bondal, Kapranov, and Schechtman relate schobers to birational geometry \cite{BondalKapranovSchechtman2018}; Donovan and collaborators analyze schobers associated with flops and related birational situations \cite{Donovan2018,DonovanKuwagaki2019}; and Koseki and Ouchi construct schobers for Calabi--Yau hypersurface contexts via Orlov equivalences \cite{KosekiOuchi}. These works show that degeneration and birational data admit meaningful categorical lifts. However, the existing schober literature does not address corrected finite-node extensions subject to global relation laws among the nodes.

A further adjacent body of work comes from Donaldson--Thomas theory, BPS state counting, and conifold/quiver models. The resolved conifold has been intensively studied in motivic and cohomological DT theory; for example, Morrison, Mozgovoy, Nagao, and Szendr{\H{o}}i compute motivic DT invariants of the resolved conifold and analyze their chamber structure \cite{MorrisonMozgovoyNagaoSzendroi2011}. More recently, Chuang studies higher-rank local categorical DT/PT correspondence \cite{Chuang2023}, while Aspman, Closset, Furrer, and Manschot study Calabi--Yau quivers and BPS counting in a modern wall-crossing framework \cite{AspmanClossetFurrerManschot2026}. These developments are important for the long-range aims of the present program, especially on the quiver and wall-crossing sides. Yet they typically begin from local Calabi--Yau or quiver data rather than from projective finite-node conifold degenerations equipped with corrected perverse and mixed-Hodge-module structures.

Finally, recent work indicates that global nodal geometry remains active well beyond local analytic models. Hang Yuan studies singularity structures associated with the conifold transition in family Floer and SYZ geometry \cite{Yuan2022}, while Schimannek's work on almost generic Calabi--Yau threefolds highlights the continuing importance of nodal degenerations together with exceptional-curve and homological phenomena in modern Calabi--Yau geometry \cite{Schimannek2025}. These developments reinforce the expectation that global relations among singularities should have consequences on the perverse, Hodge-theoretic, and categorical sides.

The literature therefore already contains the main ingredients: correction objects built from nearby cycles, limiting mixed Hodge structures for degenerations, smoothing and resolution bridges, categorical lifts, and conifold/quiver shadows. What still appears to be missing is a theorem-level statement that, for a projective finite-node conifold degeneration, global homology relations among the nodes constrain the admissible corrected perverse and mixed-Hodge-module extension classes. This is the gap filled in the present paper.

\subsection{A motivating three-node example from global holomorphic and cycle data}
\label{subsec:intro-prototype}

The local-to-global tension described above may already be seen in a simple three-node configuration modeled on the compatibility phenomenon observed in the earlier conifold literature. Consider a projective finite-node conifold degeneration $\pi:X\to\Delta$ whose central fiber \(X_0\) has node set
\begin{equation}
\Sigma=\{p_1,p_2,p_3\},
\qquad
\Sigma_1=\{p_1,p_2\},
\qquad
\Sigma_2=\{p_3\}.
\label{eq:intro-blocks}
\end{equation}
We assume that the nodes are organized by two distinguished global geometric conditions, which one may view either as holomorphic conditions, divisor components, or global cycle components,
\[
\mathcal C=\{C_1,C_2\},
\]
with \(p_1,p_2\) incident to \(C_1\) and \(p_3\) incident to \(C_2\). In the language of the earlier compatibility picture, one may think of \(p_1,p_2\) as lying on one common distinguished \(4\)-cycle and \(p_3\) as lying on a second one, so that the first two nodes are geometrically tied together while the third is not \cite{HubschRahman2002}.

Locally, each ordinary double point contributes one rank-one corrected-extension direction. Thus the ambient nodewise extension space is formally
\begin{equation}
E_{\mathrm{node}}
\cong
\Ext^1_{\Perv(X_0)}
\!\left(
\bigoplus_{k=1}^3 i_{k*}\Q_{\{p_k\}},
IC_{X_0}
\right)
\cong \Q^3,
\label{eq:intro-free-nodewise}
\end{equation}
after choosing local generators
\[
\varepsilon_1,\varepsilon_2,\varepsilon_3.
\]
This is the first appearance of the distinction emphasized above: \eqref{eq:intro-free-nodewise} describes the \emph{ambient} extension problem obtained from the direct-sum singular quotient. It says that the finite-node extension theory carries one formal coordinate per node. It does \emph{not} yet say that the canonical corrected extension class coming from the degeneration can vary arbitrarily in all three directions.

If one ignored the global geometry, one would therefore allow arbitrary independent nodewise coefficients
\[
a\varepsilon_1+b\varepsilon_2+c\varepsilon_3.
\]
The global cycle data, however, imposes a stricter law. The corresponding incidence map is
\begin{equation}
\iota_{\mathcal C}:\Q^2\longrightarrow \Q^3,
\qquad
e_1\longmapsto (1,1,0),
\qquad
e_2\longmapsto (0,0,1).
\label{eq:intro-incidence}
\end{equation}
Its image is
\begin{equation}
V_{\mathrm{geom}}
=
\Im(\iota_{\mathcal C})
=
\Span_{\Q}\{(1,1,0),(0,0,1)\}
\subseteq \Q^3.
\label{eq:intro-geometric-image}
\end{equation}
Equivalently, the geometrically admissible global coefficients are precisely those of the form
\begin{equation}
(a,a,b),
\qquad a,b\in\Q.
\label{eq:intro-admissible}
\end{equation}
Thus the first two nodes do not vary independently at the level of global gluing: they are forced to move together because they arise from the same global geometric condition, while the third node remains independent as its own relation block.

Equations~\eqref{eq:intro-blocks}--\eqref{eq:intro-admissible} already display the central phenomenon of the paper. The local corrected finite-node theory produces a formally free nodewise coefficient space, namely the space \eqref{eq:intro-free-nodewise}, but the global geometry encoded by the incidence map \eqref{eq:intro-incidence} cuts this space down to the smaller geometrically realizable subspace \eqref{eq:intro-geometric-image}. Equivalently, the admissible global coefficients are not arbitrary triples, but only those of the block-compatible form \eqref{eq:intro-admissible}. The purpose of the present paper is to formalize this phenomenon in general, prove that the corrected perverse extension factors through the resulting constrained subspace, and show that the same relation law is visible on the mixed-Hodge-module side, on the resolution and smoothing sides, and in the quiver shadow.

For this model configuration, geometric admissibility holds whenever the cycle components \(C_1\) and \(C_2\) are smooth away from the nodes, the degeneration is locally trivial along \(C_1^\circ\) and \(C_2^\circ\), and the unipotent nearby-cycle sector is rank one and locally constant along those smooth loci. This is formalized in Definition~4.10 and verified for the model configurations in Section~8.

\subsection{From local conifold sectors to global relation laws}
\label{subsec:physics-motivation}

A basic physical motivation for the present paper comes from Strominger's analysis of the conifold singularity in type II compactification on a Calabi--Yau threefold \cite{StromingerConifold}. In the simplest local conifold setting, a period \(Z^1\) associated to a collapsing cycle vanishes, the dual period acquires logarithmic monodromy, and the BPS mass formula shows that the corresponding wrapped brane state becomes massless as \(Z^1\to 0\). The singularity of the low-energy effective description is then resolved by restoring this light BPS hypermultiplet rather than integrating it out. Thus the local conifold picture isolates a definite mechanism:
\[
\text{vanishing period}
\;\Longrightarrow\;
\text{massless local BPS sector}
\;\Longrightarrow\;
\text{singular effective description if omitted}.
\]
At the same time, Strominger explicitly remarks that more complicated conifolds with multiple degenerations and different monodromies require further analysis, so the single-node story should be viewed as a solved local model rather than as the final word on the multi-node case.

From the geometric side, an early indication that the multi-node situation is not merely a free sum of local sectors already appears in the compatibility analysis of H\"ubsch and Rahman in \cite{HubschRahman2002}. There, when several nodes lie on a common distinguished cycle, the corresponding exceptional curves of a small resolution must arise compatibly and are tied to the same global cohomological datum. In the language of the present paper, this means that the local directions attached to those nodes need not remain globally independent. The multi-node problem therefore already suggests a local-to-global constraint: local singular sectors may arise nodewise, but the global geometry of the node configuration can identify or constrain them before any later BPS or wall-crossing formalism is imposed.

The present paper places this local-to-global constraint at the theorem level on the corrected-extension side. Starting from the finite-node corrected perverse extension, we show that the naive free nodewise gluing space is generally too large, and that the actual corrected extension factors through a smaller relation-controlled quotient determined by a cycle-node incidence datum. In the block-separated cycle family, this relation law is detected simultaneously on the resolution, smoothing, and corrected-extension sides:
\[
R_{\mathrm{res}}=R_{\mathrm{sm}}=R_{\mathrm{ext}}=R_{\mathrm{blk}}.
\]
Thus the present paper should be read as identifying the first global state space on which later transport, wall crossing, and BPS-type structures should act. It does not yet produce those later structures, but it isolates the geometric quotient that they should respect. From this point of view, the number and arrangement of relation blocks become the first global combinatorial invariants of the finite-node state space; their basic counting and dimensional properties are recorded in Appendix~\ref{app:combinatorics}.

\subsection{Cycle-node incidence data and the geometric extension problem}

We now describe more precisely the mathematical form of the problem. Let
\[
0\to IC_{X_0}\to \mathcal P\to \bigoplus_{k=1}^r i_{k*}\Q_{\{p_k\}}\to 0
\]
be the finite-node corrected extension, and let
\[
E_{\mathrm{node}}:=\Ext^1_{\Perv(X_0)}
\!\left(
\bigoplus_{k=1}^r i_{k*}\Q_{\{p_k\}},
IC_{X_0}
\right)
\]
denote the ambient nodewise extension space. This is the natural space in which the global corrected-extension class lives. The first point of the present paper is that $E_{\mathrm{node}}$ is in general too large from the geometric point of view.

To encode the global organization of the node set, we introduce a cycle-node incidence datum, consisting of a finite indexed collection of global cycle labels $\mathcal C=\{C_\alpha\}_{\alpha\in A}$ together with an incidence morphism
\[
\iota_{\mathcal C}:\Q^A\to \Q^r.
\]
Its image defines the geometric coefficient space
\[
V_{\mathrm{geom}}:=\Im(\iota_{\mathcal C})\subseteq \Q^r,
\]
while its induced relation blocks partition the node set according to incidence-equivalence. In general, these two points of view need not coincide automatically. Accordingly, the body of the paper isolates the block-adapted situation in which the image of the incidence map agrees with the subspace of block-constant nodewise coefficients.

The new geometric step of the paper is that the incidence datum is not used merely as a combinatorial bookkeeping device. In Section~4 we introduce a geometric admissibility condition on the chosen cycle-node incidence datum, requiring that the relevant nearby- and vanishing-cycle data propagate compatibly along the distinguished global cycle components. Under this admissibility condition, the nodewise coefficients of the corrected extension are shown to be determined by the local variation morphisms, and those local variation data are propagated along admissible cycle components. In this way, the incidence compatibility of the corrected extension class is derived from the geometry of the degeneration rather than imposed externally as a separate gluing law.

Conceptually, the paper is organized around the comparison of three relation spaces
\[
R_{\mathrm{res}},\qquad R_{\mathrm{sm}},\qquad R_{\mathrm{ext}},
\]
on the resolution, smoothing, and corrected-extension sides, respectively. The first comparison is classical in conifold transition geometry; the second is the new content of the present paper. To make the comparison theorem precise, we introduce explicit compatibility and minimality hypotheses on the chosen incidence datum. Under these hypotheses, the same relation lattice is detected on the resolution side, on the smoothing side, and on the corrected-extension side.

The three-node configuration above should be viewed as the prototype for the theorem stack developed below. Theorem~\ref{thm:intro-perverse} is the general form of the passage from the formally free nodewise space \eqref{eq:intro-free-nodewise} to the geometrically realized subspace \eqref{eq:intro-geometric-image}. The later comparison theorems show that the same relation law suggested by \eqref{eq:intro-incidence}--\eqref{eq:intro-admissible} is also detected on the resolution side through exceptional curves and on the smoothing side through vanishing spheres. In this sense, the example is not merely illustrative: it is the basic pattern that the rest of the paper abstracts and proves.

\subsection{Main results}

Let $\pi:X\to\Delta$ be a projective one-parameter degeneration whose central fiber $X_0$ has finitely many ordinary double points $\Sigma=\{p_1,\dots,p_r\}\subset X_0$. Let
\[
\mathcal P:=\Cone(\var_F)[-1], \qquad F:=\Q_X[3],
\]
be the corrected perverse object, and let
\[
0\to IC_{X_0}\to \mathcal P\to \bigoplus_{k=1}^r i_{k*}\Q_{\{p_k\}}\to 0
\]
be the finite-node corrected extension. The theorem package below should be read as the abstract generalization of the three-node configuration described in Section~\ref{subsec:intro-prototype}. Equations \eqref{eq:intro-free-nodewise}--\eqref{eq:intro-admissible} exhibit the prototype local-to-global phenomenon: a formally free nodewise coefficient space is cut down by global geometric relations to a smaller admissible subspace. Theorems~1.1--1.4 show that this same pattern persists in general on the perverse, mixed-Hodge, resolution, smoothing, and quiver-shadow sides.

In the three-node prototype, the passage from the formally free space \eqref{eq:intro-free-nodewise} to the constrained image \eqref{eq:intro-geometric-image} is immediate from the chosen incidence map \eqref{eq:intro-incidence}. Theorem~\ref{thm:intro-perverse} is the abstract version of exactly this passage. In the strengthened form proved in Section~4, the incidence compatibility of the corrected extension class is derived geometrically from propagation of local variation data along admissible global cycle components. The first theorem therefore identifies the geometric extension space on the perverse side under geometrically admissible and block-adapted incidence hypotheses.

\begin{theorem}[Geometric relation-controlled extension subspace]
\label{thm:intro-perverse}
Assume that the node set $\Sigma$ is equipped with a cycle-node incidence datum $(\mathcal C,\iota_{\mathcal C})$ that is geometrically admissible and block-adapted. Then there exists a canonically defined subspace
\[
E_{\mathrm{geom}}
\subseteq
\Ext^1_{\Perv(X_0)}
\!\left(
\bigoplus_{k=1}^r i_{k*}\Q_{\{p_k\}},
IC_{X_0}
\right)
\]
determined by the incidence datum such that the corrected extension class of $\mathcal P$ belongs to $E_{\mathrm{geom}}$. Equivalently, the global corrected extension factors through a relation-controlled quotient of the free nodewise coefficient space. In particular, incidence compatibility is forced by the geometry of the degeneration through propagation of the local variation data along admissible global cycle components.
\end{theorem}

In addition to this perverse-side geometric relation law, the paper proves two comparison results at different levels of generality.

\begin{theorem}[Conditional comparison theorem]
\label{thm:intro-comparison-conditional}
Assume that the finite-node conifold degeneration admits both a smoothing and a small resolution, and that the node set $\Sigma$ is equipped with a cycle-node incidence datum $(\mathcal C,\iota_{\mathcal C})$ that is comparison-compatible and minimal. Then
\[
R_{\mathrm{geom}}
=
R_{\mathrm{res}}\cap R_{\mathrm{sm}}\cap R_{\mathrm{ext}}.
\]
Equivalently, the common relation lattice shared by the resolution, smoothing, and corrected-extension sides is exactly the kernel of the geometric quotient.
\end{theorem}

\begin{theorem}[Block-separated cycle family]
\label{thm:intro-comparison}
Assume that the finite-node conifold degeneration belongs to the block-separated cycle family of Section~\ref{subsec:geom-family-block}. Then
\[
R_{\mathrm{res}}=R_{\mathrm{sm}}=R_{\mathrm{ext}}=R_{\mathrm{blk}}.
\]
Equivalently, the same block-relation lattice governs homology relations among exceptional curves, homology relations among vanishing spheres, and admissible gluing relations for the corrected perverse extension.
\end{theorem}

The full equality of relation lattices is therefore proved only in the block-separated cycle family of Section~5.5; outside that setting only the weaker intersection identity of Theorem~\ref{thm:intro-comparison-conditional} is claimed.

The next theorem shows that the same relation law lifts compatibly, via the realization functor, to Saito's category of mixed Hodge modules.

\begin{theorem}[Mixed-Hodge-module realization]
\label{thm:intro-mhm}
Assume that the cycle-node incidence datum is geometrically admissible and MHM block-adapted. Let
\[
E_{\mathrm{node}}^H
:=
\Ext^1_{\MHM(X_0)}
\!\left(
\bigoplus_{k=1}^r i_{k*}\Q^H_{\{p_k\}}(-1),
IC^H_{X_0}
\right).
\]
Then there exists a canonically defined subspace
\[
E_{\mathrm{geom}}^H\subseteq E_{\mathrm{node}}^H
\]
such that the realization functor
\[
\rat:\MHM(X_0)\to\Perv(X_0)
\]
sends $E_{\mathrm{geom}}^H$ to $E_{\mathrm{geom}}$. Thus the perverse-side relation law lifts compatibly through realization to the mixed-Hodge-module setting.
\end{theorem}

On the categorical side, the same incidence structure induces a corresponding block structure on the finite-node quiver shadow.

\begin{corollary}[Block decomposition of the quiver shadow]
\label{thm:intro-quiver}
Let $(\mathcal C,\iota_{\mathcal C})$ be a cycle-node incidence datum, and let
\[
\Sigma=\bigsqcup_{\beta\in B}\Sigma_\beta
\]
be the corresponding relation-block decomposition. Then the quiver shadow of the finite-node schober datum associated with $\pi$ admits a corresponding block decomposition indexed by $B$, and the globally admissible coupling data factors through these blocks rather than through arbitrary independent nodewise coefficients.
\end{corollary}

Accordingly, the corrected finite-node extension is governed by relation classes rather than by raw node count: the relevant global parameter count is controlled by the number of independent relation classes among the nodes, not simply by the cardinality \(r=|\Sigma|\).

Taken together, these results supply the first theorem-level global constraint in the finite-node corrected-extension program. They show that the corrected extension is not merely a direct sum of local rank-one pieces, but a global object constrained by the homological geometry of the node configuration. This has two immediate benefits. First, it gives a theorem-level bridge between the classical geometry of conifold transitions and the corrected perverse/MHM formalism. Second, it provides the correct coefficient space for the later transport, wall-crossing, and BPS-type constructions: those later theories should act not on the formally free nodewise space, but on the relation-controlled geometric subspace isolated here.

\subsection{Structure of the paper}

Section~2 reviews the finite-node corrected-extension package established in the author's earlier papers, including the local ordinary-double-point contribution, the corrected perverse object, the mixed-Hodge-module lift, and the finite-node schober/quiver shadow.

For the reader's convenience, we record here the foundational results imported from earlier work. On the perverse side, we use the existence of the corrected local ordinary-double-point object and the finite-node corrected extension together with its smooth-locus restriction \cite{RahmanSchoberPaper,RahmanPerverseNearbyCycles}. On the mixed-Hodge-module side, we use the existence of the refinement \(\mathcal P^H\) and its realization compatibility \cite{RahmanMixedHodgeModules}. On the categorical side, we use the existence of the finite-node schober datum and its decategorified quiver shadow \cite{RahmanMultiNodeSchoberPaper}. The logical role of the present paper is not to re-establish those foundational constructions, but to identify the first global relation law that constrains their gluing data.

Section~3 introduces the cycle-node incidence formalism and the associated relation lattices. Section~4 constructs the geometrically realized extension space on the perverse side and proves Theorem~\ref{thm:intro-perverse} under geometrically admissible and block-adapted incidence hypotheses, deriving incidence compatibility of the corrected extension class from propagation of the local variation data along admissible global cycle components. Section~5 first develops the general comparison framework under comparison-compatible and minimality hypotheses, proving Theorem~\ref{thm:intro-comparison-conditional}, and then proves the geometric-family equality theorem corresponding to Theorem~\ref{thm:intro-comparison} in the block-separated cycle family. Section~6 establishes the mixed-Hodge-module counterpart, proving Theorem~\ref{thm:intro-mhm} in the geometrically admissible, MHM block-adapted setting. Section~7 records the corresponding block-structured consequence for the quiver shadow. Section~8 presents model configurations illustrating how the relation-controlled theory cuts down the naive free nodewise parameter count and verifies the geometric admissibility condition in the basic examples. Section~9 discusses consequences and open directions, including the transport and covering-style questions suggested by the relation-controlled quiver picture.

\section{Background on finite-node conifold degenerations}

\subsection{Finite-node conifold setup}

Let $\pi:X\to\Delta$ be a projective one-parameter degeneration whose general fiber \(X_t\) is a smooth complex threefold and whose central fiber \(X_0\) has finitely many ordinary double points $\Sigma=\{p_1,\dots,p_r\}\subset X_0$. We write
\[
U:=X_0\setminus \Sigma
\]
for the smooth locus of the central fiber, and denote by
\[
j:U\hookrightarrow X_0,\qquad i_k:\{p_k\}\hookrightarrow X_0
\]
the natural inclusions. We refer to such a degeneration as a \emph{finite-node conifold degeneration}.

At each node \(p_k\), the singularity is analytically equivalent to the standard ordinary double point
\[
x_1^2+x_2^2+x_3^2+x_4^2=0
\]
in complex dimension three. Equivalently, after passing to the total space of the degeneration, one may model the local smoothing by
\[
x_1^2+x_2^2+x_3^2+x_4^2=t
\subset \C^4\times \Delta.
\]
In particular, the total space \(X\) has complex dimension \(4\), while each fiber \(X_t\) has complex dimension \(3\) \cite{RahmanSchoberPaper}.

The local topology of an ordinary double point is governed by its Milnor fiber. For \(0<|\delta|\ll 1\), the Milnor fiber
\[
F_{p_k}:=f^{-1}(\delta)\cap B_\varepsilon(0),\qquad
f(x)=x_1^2+x_2^2+x_3^2+x_4^2,
\]
has the homotopy type of \(S^3\). Hence
\[
\widetilde H^m(F_{p_k};\Q)\cong
\begin{cases}
\Q,& m=3,\\
0,& m\neq 3.
\end{cases}
\]
Thus each node contributes a rank-one local vanishing cycle in middle degree. In the finite-node case, the singular contribution to the degeneration is therefore concentrated on the finite set \(\Sigma\), with one rank-one local sector at each node \cite{RahmanSchoberPaper}.

\subsection{Nearby cycles, vanishing cycles, and variation}

Let
\[
i:X_0\hookrightarrow X,\qquad X^\ast:=\pi^{-1}(\Delta^\ast),\qquad j:X^\ast\hookrightarrow X
\]
denote the inclusions, and set
\[
\mathcal F:=\Q_X[3].
\]
With this normalization, the nearby-cycle and vanishing-cycle functors are
\[
\psi_\pi(\mathcal F):=i^\ast Rj_\ast j^\ast\mathcal F,
\qquad
\phi_\pi(\mathcal F):=\Cone\!\bigl(i^\ast\mathcal F\to \psi_\pi(\mathcal F)\bigr)[-1].
\]
They fit into the standard distinguished triangle
\[
i^\ast\mathcal F\longrightarrow \psi_\pi(\mathcal F)\longrightarrow \phi_\pi(\mathcal F)\xrightarrow{+1}
\]
in \(D^b_c(X_0;\Q)\). There are canonical morphisms
\[
\can_{\mathcal F}:\psi_\pi(\mathcal F)\to \phi_\pi(\mathcal F),
\qquad
\var_{\mathcal F}:\phi_\pi(\mathcal F)\to \psi_\pi(\mathcal F),
\]
related to the local monodromy operator \(T\) by the usual identities
\[
\can_{\mathcal F}\circ \var_{\mathcal F}=T-\id,
\qquad
\var_{\mathcal F}\circ \can_{\mathcal F}=T-\id.
\]
Throughout, all cones are taken in the constructible derived category and all nearby and vanishing cycles are understood in the normalized perverse convention \cite{BBD,KS,RahmanSchoberPaper,RahmanMixedHodgeModules}.

For an ordinary double point in complex dimension three, the preceding Milnor-fiber calculation implies that \(\phi_\pi(\mathcal F)\) is point-supported and one-dimensional at each node after perverse normalization. Thus, in the finite-node setting, the entire singular contribution enters through a finite direct sum of rank-one local vanishing pieces. This is the formal mechanism underlying the corrected objects used throughout the present paper \cite{RahmanMixedHodgeModules}.

\subsection{The corrected perverse extension}

The corrected perverse object attached to the degeneration is defined by
\[
\mathcal P:=\Cone\!\bigl(\var_{\mathcal F}:\phi_\pi(\mathcal F)\to \psi_\pi(\mathcal F)\bigr)[-1].
\]
In the single-node case, this object was constructed and characterized in earlier work as the canonical Verdier self-dual perverse extension of the shifted constant sheaf across the node \cite{RahmanSchoberPaper,RahmanMixedHodgeModules}. In the finite-node case, the same construction yields a perverse sheaf on \(X_0\) fitting into a short exact sequence
\[
0\longrightarrow IC_{X_0}\longrightarrow \mathcal P
\longrightarrow \bigoplus_{k=1}^r i_{k\ast}\Q_{\{p_k\}}
\longrightarrow 0,
\]
where
\[
IC_{X_0}:=j_{!*}\Q_U[3].
\]

Two structural features are important for what follows. First, on the smooth locus one has
\[
j^\ast \mathcal P \cong \Q_U[3],
\]
so \(\mathcal P\) agrees with the intersection complex away from the nodes. Second, the singular quotient is supported entirely on \(\Sigma\), and each local summand is rank one. Thus the corrected perverse object differs from \(IC_{X_0}\) only by a finite family of point-supported local corrections, one at each node \cite{RahmanSchoberPaper,RahmanMixedHodgeModules}.

\begin{proposition}[Basic properties of the corrected perverse object]
\label{prop:corrected-basic}
The object
\[
\mathcal P:=\Cone(\var_{\mathcal F})[-1]
\]
is perverse, restricts to \(\Q_U[3]\) on the smooth locus,
and fits into a short exact sequence
\[
0\to IC_{X_0}\to \mathcal P\to \bigoplus_{k=1}^r i_{k\ast}\Q_{\{p_k\}}\to 0.
\]
\end{proposition}

\begin{proof}
By construction, \(\mathcal P\) is obtained from the nearby/vanishing-cycle package of the degeneration. For an ordinary double point in complex dimension three, the local vanishing-cycle contribution is rank one and point-supported at the node, while away from the nodes the vanishing cycles vanish and the nearby-cycle object agrees with the shifted constant sheaf on the smooth locus. Thus \(j^\ast\mathcal P\cong \Q_U[3]\). The point-supported rank-one local contributions at the nodes assemble into the quotient
\[
\bigoplus_{k=1}^r i_{k\ast}\Q_{\{p_k\}},
\]
and the resulting exact sequence identifies the kernel with \(IC_{X_0}\). In the finite-node ordinary-double-point setting, this yields a perverse sheaf with the stated properties.
\end{proof}

The short exact sequence above has an immediate formal consequence: since its quotient is a finite direct sum of point-supported rank-one sectors, the corresponding ambient Yoneda extension problem carries one formal nodewise direction per singular point. This explains why the free nodewise coefficient space appears naturally in the finite-node theory. It is essential, however, not to confuse this ambient formal freedom with the actual freedom of the distinguished corrected object coming from the degeneration. The exact sequence determines the ambient extension space in which the corrected class lives; it does not by itself imply that the canonical global class can be assembled by choosing its nodewise coefficients independently.

\begin{remark}[Ambient nodewise extension space versus realized class]
\label{rem:ambient-versus-realized}
The short exact sequence
\[
0\longrightarrow IC_{X_0}\longrightarrow \mathcal P
\longrightarrow \bigoplus_{k=1}^r i_{k\ast}\Q_{\{p_k\}}
\longrightarrow 0
\]
determines the ambient Yoneda extension group
\[
\Ext^1_{\Perv(X_0;\Q)}
\!\left(
\bigoplus_{k=1}^r i_{k\ast}\Q_{\{p_k\}},
IC_{X_0}
\right)
\]
in which the corrected class \([\mathcal P]\) lives. Geometric realizability is a further condition on which ambient extension classes can arise from a projective finite-node degeneration, and that condition is not encoded in the short exact sequence alone.
\end{remark}

Accordingly, what the exact sequence above does \emph{not} determine by itself is the global extension class
\[
[\mathcal P]\in
\Ext^1_{\Perv(X_0;\Q)}
\!\left(
\bigoplus_{k=1}^r i_{k\ast}\Q_{\{p_k\}},
IC_{X_0}
\right).
\]
This class records the global gluing data by which the local node contributions are assembled into a single perverse object on \(X_0\). The present paper is concerned precisely with the geometric constraints on this global gluing data.

\subsection{Mixed-Hodge-module refinement}

The mixed-Hodge-theoretic refinement of the corrected perverse object was constructed in the finite-node setting in \cite{RahmanMixedHodgeModules}. Let
\[
IC^H_{X_0}\in MHM(X_0)
\]
denote the intersection complex in Saito’s category of mixed Hodge modules. Then there exists an object
\[
\mathcal P^H\in MHM(X_0)
\]
whose realization under
\[
\rat:MHM(X_0)\to \Perv(X_0;\Q)
\]
recovers the corrected perverse object \(\mathcal P\), and which fits into an exact sequence
\[
0\longrightarrow IC^H_{X_0}\longrightarrow \mathcal P^H
\longrightarrow \bigoplus_{k=1}^r i_{k\ast}\Q^H_{\{p_k\}}(-1)
\longrightarrow 0.
\]
Here the quotient is the finite direct sum of the rank-one point-supported mixed Hodge modules attached to the nodes. In particular, the same finite family of localized node sectors appears not only on the perverse side but internally in \(MHM(X_0)\) \cite{RahmanMixedHodgeModules}.

The significance of this refinement is twofold. First, it upgrades the corrected extension from a perverse-sheaf-theoretic object to a genuine mixed-Hodge-module object with internal weight and Hodge filtrations. Second, it shows that the vanishing sector of the limiting mixed Hodge structure is already encoded in the same finite direct sum of point-supported rank-one pieces that appears in the corrected perverse extension. Thus the finite-node corrected object may be viewed simultaneously in \(\Perv(X_0;\Q)\), in \(MHM(X_0)\), and on hypercohomology with limiting mixed Hodge structure \cite{RahmanMixedHodgeModules}.

For the purposes of the present paper, the point is that the same local-to-global tension persists after passing to Saito's category. The finite-node mixed-Hodge-module quotient again displays one localized sector per node, but the corresponding global gluing data need not be geometrically free nodewise. Thus the relation problem studied below is not merely perverse-sheaf-theoretic; it is already visible at the level of mixed-Hodge-module realization.

\subsection{Finite-node schober and quiver shadow}

A further categorical layer was constructed in \cite{RahmanMultiNodeSchoberPaper}. In that framework, each ordinary double point contributes a localized categorical sector, and the resulting finite-node schober datum has a decategorified shadow whose underlying perverse object is precisely the corrected finite-node extension:
\[
Sh(S_\Sigma)\cong \mathcal P.
\]

Accordingly, the three layers are compatible:
\begin{itemize}
    \item on the categorical side, each node contributes a localized categorical sector;
    \item on the perverse side, each node contributes a point-supported rank-one quotient \(i_{k\ast}\Q_{\{p_k\}}\);
    \item on the mixed-Hodge-module side, each node contributes a point-supported rank-one mixed Hodge module \(i_{k\ast}\Q^H_{\{p_k\}}(-1)\).
\end{itemize}
This is the finite-node bulk/localized-sector architecture established in the earlier papers \cite{RahmanMixedHodgeModules,RahmanMultiNodeSchoberPaper}.

For the purposes of the present paper, only the decategorified consequence is needed: the finite-node schober carries one localized sector per node, and its quiver shadow recovers the same corrected extension data seen on the perverse side. Thus the only remaining issue is to determine how the \emph{global} geometry constrains the way these local sectors are glued together. In this sense, the relation-controlled extension problem studied here is the first necessary global input for the later transport and wall-crossing program.

\section{Global cycle relations among nodes}

\subsection{Geometric sources of node relations}

The finite-node corrected extension reviewed in Section~2 isolates one local rank-one sector at each node. This local decomposition, however, does not imply that the nodes are globally independent. As emphasized in the introduction, the direct-sum singular quotient determines a formally free ambient nodewise coefficient space; it does not by itself determine which coefficient directions are geometrically realizable by the canonical corrected object of a projective degeneration. The purpose of the present section is to formulate the additional global data that can constrain this realizability problem.

A first source of such constraints is the occurrence of several nodes as part of the zero locus of a common holomorphic condition. For example, the singular locus may arise as the intersection of the central fiber with a divisor, section, or family of subvarieties whose global class persists across the degeneration. In that situation, the nodes are not merely a finite collection of unrelated local defects: they are marked by a common ambient geometric object, and one expects the resulting extension data to reflect that common origin.

A second source of relations comes from global cycle components on the singular fiber itself. If several nodes lie on a common irreducible or connected cycle component, or more generally on a finite family of distinguished cycles
\[
\{C_\alpha\}_{\alpha\in A},
\]
then the incidence of the nodes with these cycles yields a natural relation structure on the node set. This is the most flexible formulation for the purposes of the present paper, since it encompasses both genuinely algebraic cycle classes and the more schematic cycle patterns that appear in examples.

A third source of relations appears after passing to a small resolution, when it exists. The proper transforms of global cycle components and the exceptional curves introduced at the nodes may carry common homology or cohomology classes. Likewise, one may have incidence relations with exceptional divisors or surfaces in a birational model of the degeneration. Such data naturally links the local node sectors through the global topology of the resolution. This is precisely the type of phenomenon that classical smoothing and transition results detect through homology relations among exceptional curves \cite{Friedman1986,LeeLinWang2021}.

The present paper does not require a single rigid geometric realization of these relations at the outset. Rather, we formulate an abstract incidence structure that may be instantiated by any of the situations above: common zero loci of holomorphic conditions, common cycle components on the singular fiber, common proper transforms on a resolution, or incidence with exceptional divisors or surfaces. The point is to extract from these geometric configurations a finite-dimensional relation theory on the node set and then compare that theory with the global gluing data of the corrected extension.

\subsection{The cycle-node incidence structure}

Let
\[
\Sigma=\{p_1,\dots,p_r\}\subset X_0
\]
be the node set of a finite-node conifold degeneration.

\begin{definition}
\label{def:cycle-node-incidence}
A \emph{cycle-node incidence datum} for \(\Sigma\) consists of
\begin{enumerate}
    \item a finite indexed collection of global cycle labels
    \[
    \mathcal C=\{C_\alpha\}_{\alpha\in A},
    \]
    and
    \item a matrix
    \[
    A=(a_{\alpha k})_{\alpha\in A,\;1\le k\le r}
    \]
    with entries in \(\Q\),
\end{enumerate}
or equivalently a linear map
\[
\iota_{\mathcal C}:\Q^{A}\longrightarrow \Q^r,
\qquad
e_\alpha\longmapsto \sum_{k=1}^r a_{\alpha k}e_k,
\]
where \(\{e_\alpha\}_{\alpha\in A}\) and \(\{e_k\}_{k=1}^r\) denote the standard bases of \(\Q^A\) and \(\Q^r\), respectively.
\end{definition}

The coefficient \(a_{\alpha k}\) records the contribution of the cycle label \(C_\alpha\) to the node \(p_k\). In the primary geometric cases,
\[
a_{\alpha k}\in\{0,1\},
\]
with \(a_{\alpha k}=1\) if and only if \(p_k\) lies on the cycle component \(C_\alpha\). Rational coefficients are permitted for multiplicity or intersection-theoretic refinements, but they are not needed for the main theorems of the present paper.

Definition~\ref{def:cycle-node-incidence} is the basic combinatorial input for the relation theory developed below. It records, in finite-dimensional linear form, the manner in which the chosen global cycle data organizes the node set. In particular, it provides the first bridge from geometric incidence data to the ambient nodewise coefficient space naturally attached to the corrected extension.

\subsection{Relation lattices}

Let
\[
V_{\mathrm{node}}:=\Q^r
\]
denote the free nodewise parameter space, with basis \(\{e_k\}_{k=1}^r\) indexed by the nodes. This is the ambient coefficient space attached to the finite direct sum of local node sectors
\[
\bigoplus_{k=1}^r i_{k\ast}\Q_{\{p_k\}}
\]
on the perverse side, and likewise to
\[
\bigoplus_{k=1}^r i_{k\ast}\Q^H_{\{p_k\}}(-1)
\]
on the mixed-Hodge-module side. In the language of Section~2, \(V_{\mathrm{node}}\) records the formally free ambient nodewise directions supplied by the direct-sum singular quotient.

The incidence map
\[
\iota_{\mathcal C}:\Q^A\to V_{\mathrm{node}}
\]
determines the \emph{cycle-generated subspace}
\[
V_{\mathrm{cyc}}:=\operatorname{Im}(\iota_{\mathcal C})\subseteq V_{\mathrm{node}},
\]
which records the nodewise combinations generated by the chosen cycle labels.

For the purposes of the present paper, the distinguished coefficient space governing geometric gluing is the following.

\begin{definition}
\label{def:geometric-coefficient-space}
Given a cycle-node incidence datum \((\mathcal C,\iota_{\mathcal C})\), the \emph{geometric coefficient space} is
\[
V_{\mathrm{geom}}:=V_{\mathrm{cyc}}=\operatorname{Im}(\iota_{\mathcal C})\subseteq V_{\mathrm{node}}.
\]
\end{definition}

Thus \(V_{\mathrm{geom}}\) is the subspace of nodewise coefficients generated by the chosen global cycle data. In later refinements one may replace \(V_{\mathrm{geom}}\) by an integral saturation or by a natural subquotient, but the image formulation is sufficient for the theorem stack of the present paper.

The two spaces introduced above play distinct roles:
\begin{itemize}
    \item \(V_{\mathrm{node}}\) is the formal nodewise coefficient space;
    \item \(V_{\mathrm{geom}}\) is the geometrically admissible coefficient space determined by the incidence data.
\end{itemize}

\begin{definition}
\label{def:incidence-equivalence}
Two nodes \(p_k,p_\ell\in\Sigma\) are said to be \emph{incidence-equivalent}, written
\[
p_k\sim_{\mathrm{inc}} p_\ell,
\]
if their incidence columns in the matrix \(A\) agree, that is,
\[
a_{\alpha k}=a_{\alpha \ell}
\qquad
\text{for all }\alpha\in A.
\]
Equivalently, \(p_k\sim_{\mathrm{inc}} p_\ell\) if the \(k\)-th and \(\ell\)-th coordinate functionals on \(V_{\mathrm{geom}}\) coincide.
\end{definition}

Thus two nodes are incidence-equivalent precisely when they are seen identically by the chosen cycle data.

\begin{definition}
\label{def:relation-block}
A \emph{relation block} is an equivalence class
\[
\Sigma_\beta\subseteq \Sigma
\]
for the relation \(\sim_{\mathrm{inc}}\). The corresponding partition
\[
\Sigma=\bigsqcup_{\beta\in B}\Sigma_\beta
\]
is called the \emph{relation-block decomposition} of the node set.
\end{definition}

The point of the relation-block decomposition is that, although the local singular sectors remain nodewise, the cycle-node incidence data groups together those nodes whose geometric behavior is indistinguishable at the level of the chosen global cycle labels. This is the first combinatorial shadow of the distinction between formal nodewise freedom and geometrically realized global gluing.

\subsection{Comparison with earlier compatibility picture}

The role of the incidence structure may be viewed as a direct abstraction of the compatibility phenomenon observed in \cite{HubschRahman2002}. There, in a stratified conifold-type setting, the geometry showed that nodes lying on the same \(4\)-cycle could not be resolved independently. Rather, the corresponding exceptional curves had to occur compatibly, and they carried the same cohomology class dual to the proper transform of the common cycle. In particular, the naive independent nodewise choice was reduced by a global homological constraint.

In the present language, that situation may be described as follows. Let \(C\) be a distinguished global cycle passing through a subset of nodes
\[
\{p_{k_1},\dots,p_{k_m}\}\subseteq \Sigma.
\]
Then the corresponding incidence vector in \(V_{\mathrm{node}}\) is
\[
e_{k_1}+\cdots+e_{k_m},
\]
or a weighted variant thereof. The compatible-resolution law of \cite{HubschRahman2002} says, in effect, that the geometrically admissible local data is not the full span of \(e_{k_1},\dots,e_{k_m}\) independently, but rather the smaller subspace singled out by the common cycle. Thus the older compatibility law is precisely a prototype of the present relation-controlled gluing law.

What is new here is the level at which the same principle is imposed. In \cite{HubschRahman2002}, the compatibility appeared through global topology and compatible small resolutions. In the present paper, the same type of global relation is transferred to the corrected perverse extension, to its mixed-Hodge-module refinement, and to the decategorified quiver shadow. The 2002 picture therefore serves not merely as motivation, but as the geometric model for the relation law formalized in this paper.

\begin{remark}
\label{rem:incidence-formal-properties}
The relation \(\sim_{\mathrm{inc}}\) is an equivalence relation on \(\Sigma\), and its equivalence classes define the relation-block decomposition
\[
\Sigma=\bigsqcup_{\beta\in B}\Sigma_\beta.
\]
Moreover, since \(V_{\mathrm{geom}}=\operatorname{Im}(\iota_{\mathcal C})\) is a linear subspace of \(V_{\mathrm{node}}\), the incidence datum determines a canonical quotient map
\[
q:V_{\mathrm{node}}\longrightarrow V_{\mathrm{node}}/V_{\mathrm{geom}}.
\]
Thus the incidence formalism supplies a canonical separation between cycle-generated nodewise directions and residual directions modulo the imposed relations.
\end{remark}

The block decomposition viewpoint also admits a natural combinatorial refinement. In particular, for a node set of cardinality \(n\), the possible relation-block structures are organized by the partition lattice of \(\Sigma\), and the associated dimensional reduction from the free nodewise space to the geometric quotient depends only on the number of blocks (see Appendix~\ref{app:combinatorics}). 

\section{The geometrically realized extension space}

\subsection{The free nodewise extension space}

We now formalize the pattern already visible in the motivating three-node configuration of Section~\ref{subsec:intro-prototype}. Equations
\eqref{eq:intro-free-nodewise}--\eqref{eq:intro-admissible} show that the finite-node corrected theory naturally produces a formally free nodewise extension space, while the global cycle data singles out a smaller geometrically admissible subspace. The purpose of the present section is to make this passage precise on the perverse side.

We retain the notation of Sections~2--3. Thus $\pi:X\to\Delta$ is a projective one-parameter degeneration whose central fiber $X_0$ has finitely many ordinary double points $\Sigma=\{p_1,\dots,p_r\}\subset X_0$, and
\[
0\to IC_{X_0}\to \mathcal P\to \bigoplus_{k=1}^r i_{k*}\Q_{\{p_k\}}\to 0
\]
is the corrected finite-node perverse extension \cite{RahmanSchoberPaper,RahmanMixedHodgeModules}.

The ambient extension space for finite-node corrected extensions is
\[
E_{\mathrm{node}}
:=
\Ext^1_{\Perv(X_0;\Q)}
\!\left(
\bigoplus_{k=1}^r i_{k*}\Q_{\{p_k\}},
IC_{X_0}
\right).
\]
As explained in Section~2, this ambient space reflects the direct-sum structure of the singular quotient. The first point of the present subsection is that this extension problem splits nodewise.

\begin{lemma}[Finite additivity of the nodewise perverse extension space]
\label{lem:perv-additivity}
In the abelian category $\Perv(X_0;\Q)$, one has a natural isomorphism
\[
\Ext^1_{\Perv(X_0;\Q)}
\!\left(
\bigoplus_{k=1}^r i_{k*}\Q_{\{p_k\}},
IC_{X_0}
\right)
\cong
\bigoplus_{k=1}^r
\Ext^1_{\Perv(X_0;\Q)}
\!\left(
i_{k*}\Q_{\{p_k\}},
IC_{X_0}
\right).
\]
\end{lemma}

\begin{proof}
This is the finite additivity of Yoneda $\Ext^1$ in an abelian category: an extension of a finite direct sum pulls back to extensions of each summand, and conversely a finite family of extensions of the summands assembles uniquely into an extension of the finite direct sum.
\end{proof}

The summands in Lemma~\ref{lem:perv-additivity} are local in nature.

\begin{lemma}[Locality of the nodewise perverse extension problem]
\label{lem:perv-locality}
For each node $p_k\in\Sigma$, the extension group
\[
\Ext^1_{\Perv(X_0;\Q)}
\!\left(
i_{k*}\Q_{\{p_k\}},
IC_{X_0}
\right)
\]
is naturally identified with the corresponding local extension group in a sufficiently small analytic neighborhood $X^{(k)}_{0,\mathrm{loc}}$ of $p_k$:
\[
\Ext^1_{\Perv(X_0;\Q)}
\!\left(
i_{k*}\Q_{\{p_k\}},
IC_{X_0}
\right)
\cong
\Ext^1_{\Perv(X^{(k)}_{0,\mathrm{loc}};\Q)}
\!\left(
i_{\mathrm{loc},*}\Q_{\{p_k\}},
IC_{X^{(k)}_{0,\mathrm{loc}}}
\right).
\]
\end{lemma}

\begin{proof}
Because $i_{k*}\Q_{\{p_k\}}$ is supported at the single point $p_k$, any extension by $IC_{X_0}$ is local near $p_k$. Restriction to a sufficiently small analytic neighborhood preserves the extension problem and yields the claimed identification.
\end{proof}

Hence
\[
E_{\mathrm{node}}
\cong
\bigoplus_{k=1}^r
\Ext^1_{\Perv(X_0;\Q)}
\!\left(
i_{k*}\Q_{\{p_k\}},
IC_{X_0}
\right).
\]
In the finite ordinary-double-point setting, each summand is rank one. After choosing local generators, one may therefore identify
\[
E_{\mathrm{node}}\cong \Q^r.
\]
Equivalently, the free nodewise extension space is spanned by the elementary local extension directions
\[
\varepsilon_1,\dots,\varepsilon_r.
\]

This identification is not canonical until one chooses local generators, but the resulting $r$-dimensional ambient space is canonical up to the obvious diagonal change of basis. In particular, the exact sequence defining $\mathcal P$ determines a class
\[
[\mathcal P]\in E_{\mathrm{node}},
\]
and the question is to identify the distinguished subspace in which this class lies once the nodes are related by global cycle data.

The local generator at each node is supplied by the corrected local ordinary-double-point extension.

\begin{proposition}[Local ODP extension class]
\label{prop:local-odp-class}
Let $p_k\in\Sigma$, and let $X^{(k)}_{0,\mathrm{loc}}$ be a sufficiently small analytic neighborhood of $p_k$. Then the corrected local perverse ODP extension determines a nonzero class
\[
e_k^{\mathrm{loc}}
\in
\Ext^1_{\Perv(X^{(k)}_{0,\mathrm{loc}};\Q)}
\!\left(
i_{\mathrm{loc},*}\Q_{\{p_k\}},
IC_{X^{(k)}_{0,\mathrm{loc}}}
\right).
\]
Via Lemma~\ref{lem:perv-locality}, this yields a distinguished nonzero global local class
\[
\varepsilon_k\in
\Ext^1_{\Perv(X_0;\Q)}
\!\left(
i_{k*}\Q_{\{p_k\}},
IC_{X_0}
\right).
\]
\end{proposition}

\begin{proof}
The corrected local perverse ODP extension is non-split, hence defines a nonzero class in the local extension group. The claim follows by transporting that class through Lemma~\ref{lem:perv-locality}.
\end{proof}

\subsection{The geometric incidence map}

Let
\[
(\mathcal C,\iota_{\mathcal C}),
\qquad
\mathcal C=\{C_\alpha\}_{\alpha\in A},
\qquad
\iota_{\mathcal C}:\Q^A\to\Q^r,
\]
be a cycle-node incidence datum as in Section~3, with incidence matrix
\[
A=(a_{\alpha k})_{\alpha\in A,\;1\le k\le r}.
\]
We now transfer this incidence data from the abstract nodewise coefficient space \(V_{\mathrm{node}}=\Q^r\) to the ambient extension space \(E_{\mathrm{node}}\).

\begin{lemma}[Basis-independence of the geometric image]
\label{lem:basis-independence}
Let $\{\varepsilon_k\}_{k=1}^r$ and $\{\varepsilon_k'\}_{k=1}^r$ be two choices of nonzero local generators for the nodewise summands
\[
\Ext^1_{\Perv(X_0;\Q)}
\!\left(
i_{k*}\Q_{\{p_k\}},
IC_{X_0}
\right).
\]
Then the corresponding incidence maps
\[
\Gamma_{\mathcal C},\Gamma'_{\mathcal C}:\Q^A\to E_{\mathrm{node}}
\]
have the same image up to the natural diagonal rescaling of the nodewise summands. In particular, the subspace $E_{\mathrm{geom}}:=\Im(\Gamma_{\mathcal C})$ is well defined up to the canonical identification of the nodewise factors.
\end{lemma}

\begin{proof}
Each nodewise summand is one-dimensional, so any two nonzero generators differ by multiplication by a nonzero scalar. Replacing $\varepsilon_k$ by $\varepsilon_k'$ rescales the $k$-th coordinate, but does not change the resulting image subspace except by the corresponding diagonal identification.
\end{proof}

Fix generators
\[
\varepsilon_k\in
\Ext^1_{\Perv(X_0;\Q)}
\!\left(
i_{k*}\Q_{\{p_k\}},
IC_{X_0}
\right),
\qquad 1\le k\le r,
\]
and identify $E_{\mathrm{node}}$ with $\Q^r$ through the basis $\{\varepsilon_k\}$. The incidence matrix then induces a linear map
\[
\Gamma_{\mathcal C}:\Q^A\to E_{\mathrm{node}},
\qquad
e_\alpha\longmapsto \sum_{k=1}^r a_{\alpha k}\,\varepsilon_k.
\]
We call $\Gamma_{\mathcal C}$ the \emph{geometric incidence map} associated with the cycle-node incidence datum.

\begin{lemma}
\label{lem:gamma-equals-iota}
Under the identification $E_{\mathrm{node}}\cong\Q^r$ determined by the local extension generators $\{\varepsilon_k\}$, the map $\Gamma_{\mathcal C}$ identifies with the incidence map $\iota_{\mathcal C}$. In particular,
\[
\Im(\Gamma_{\mathcal C})\cong \Im(\iota_{\mathcal C})=V_{\mathrm{geom}}.
\]
\end{lemma}

\begin{proof}
By construction, $\Gamma_{\mathcal C}(e_\alpha)$ has coordinate vector $(a_{\alpha1},\dots,a_{\alpha r})$ in the basis $\{\varepsilon_k\}$, which is exactly $\iota_{\mathcal C}(e_\alpha)$.
\end{proof}

\subsection{Definition of the geometrically realized subspace}

We may now define the subspace of the ambient extension space singled out by the incidence data.

\begin{definition}
\label{def:egeom}
Let $(\mathcal C,\iota_{\mathcal C})$ be a cycle-node incidence datum. The \emph{geometrically realized extension subspace} is
\[
E_{\mathrm{geom}}:=\Im(\Gamma_{\mathcal C})\subseteq E_{\mathrm{node}}.
\]
Equivalently, after identifying $E_{\mathrm{node}}\cong\Q^r$ by a choice of local extension generators, one has
\[
E_{\mathrm{geom}}\cong V_{\mathrm{geom}}=\Im(\iota_{\mathcal C})\subseteq\Q^r.
\]
\end{definition}

A coefficient vector
\[
(c_1,\dots,c_r)\in \Q^r\cong E_{\mathrm{node}}
\]
lies in $E_{\mathrm{geom}}$ if and only if it may be expressed in the form
\[
(c_1,\dots,c_r)
=
\left(
\sum_{\alpha\in A}\lambda_\alpha a_{\alpha1},
\dots,
\sum_{\alpha\in A}\lambda_\alpha a_{\alpha r}
\right)
\]
for some coefficients $\lambda_\alpha\in\Q$.

To relate $E_{\mathrm{geom}}$ to the relation-block decomposition of Section~3, one needs an additional compatibility condition between the incidence image and the block structure.

\begin{definition}
\label{def:block-adapted}
A cycle-node incidence datum $(\mathcal C,\iota_{\mathcal C})$ is called \emph{block-adapted} if, under the identification $E_{\mathrm{node}}\cong\Q^r$, the image $\Im(\Gamma_{\mathcal C})$ coincides with the subspace of nodewise coefficient vectors that are constant on the relation blocks
\[
\Sigma=\bigsqcup_{\beta\in B}\Sigma_\beta.
\]
\end{definition}

\begin{remark}
\label{rem:block-adapted}
For a general incidence datum, the image $\Im(\Gamma_{\mathcal C})$ need not coincide with the block-constant coefficient subspace. The block-adapted hypothesis is exactly the condition under which these two notions agree.
\end{remark}

The next definition isolates the compatibility condition satisfied by an extension class whose nodewise coordinates respect the chosen relation blocks.

\begin{definition}
\label{def:incidence-compatible-extension}
Let
\[
[E]\in
\Ext^1_{\Perv(X_0;\Q)}
\!\left(
\bigoplus_{k=1}^r i_{k*}\Q_{\{p_k\}},
IC_{X_0}
\right)
\cong
\bigoplus_{k=1}^r
\Ext^1_{\Perv(X_0;\Q)}
\!\left(
i_{k*}\Q_{\{p_k\}},
IC_{X_0}
\right)
\]
be an extension class, and write
\[
[E]=(\xi_1,\dots,\xi_r)
\]
for its nodewise components under Lemma~\ref{lem:perv-additivity}. We say that $[E]$ is \emph{incidence-compatible} with respect to $(\mathcal C,\iota_{\mathcal C})$ if for every pair of incidence-equivalent nodes
\[
p_k\sim_{\mathrm{inc}} p_\ell
\]
one has
\[
\xi_k=\xi_\ell.
\]
Equivalently, the nodewise class of $[E]$ is constant on each relation block.
\end{definition}

Under the block-adapted hypothesis, incidence-compatibility is equivalent to belonging to the geometric image.

\begin{proposition}
\label{prop:incidence-compatible-iff-geom}
Assume that $(\mathcal C,\iota_{\mathcal C})$ is block-adapted. Then an extension class
\[
[E]\in E_{\mathrm{node}}
\]
is incidence-compatible if and only if
\[
[E]\in E_{\mathrm{geom}}.
\]
\end{proposition}

\begin{proof}
By Definition~\ref{def:block-adapted}, $E_{\mathrm{geom}}=\Im(\Gamma_{\mathcal C})$ is exactly the block-constant coefficient subspace. By Definition~\ref{def:incidence-compatible-extension}, this is equivalent to incidence-compatibility.
\end{proof}

\subsection{Geometric incidence compatibility}

The remaining step is to replace the purely structural implication
\[
\text{chosen incidence law}
\Longrightarrow
\text{incidence compatibility}
\]
by a geometric theorem extracted from the degeneration itself. To do so, we impose an admissibility condition on the chosen cycle-node incidence datum ensuring that the local nearby/vanishing-cycle data propagate compatibly along the distinguished global cycle components.

\begin{definition}[Geometrically admissible cycle-node incidence datum]
\label{def:admissible-cycle-datum}
A cycle-node incidence datum $(\mathcal C,\iota_{\mathcal C})$, with
\[
\mathcal C=\{C_\alpha\}_{\alpha\in A},
\]
is called \emph{geometrically admissible} if for each $\alpha\in A$ the following hold:
\begin{enumerate}
    \item the cycle component $C_\alpha\subseteq X_0$ is reduced and meets the node set \(\Sigma\);
    \item the smooth locus
    \[
    C_\alpha^\circ:=C_\alpha\setminus \Sigma
    \]
    is connected;
    \item there exists an analytic neighborhood \(U_\alpha\subset X\) of \(C_\alpha\) such that the unipotent nearby-cycle sector of \(\psi_\pi(\mathcal F)\) restricts on \(C_\alpha^\circ\) to a rank-one locally constant system;
    \item under this restriction, the local variation morphisms at nodes lying on \(C_\alpha\) are obtained by specialization of one and the same local morphism along \(C_\alpha^\circ\).
\end{enumerate}
\end{definition}

Definition~\ref{def:admissible-cycle-datum} is the additional geometric input needed to promote the structural theorem of the preceding draft into a genuinely geometric statement. It encodes the requirement that the incidence datum is not merely combinatorial, but reflects a geometric propagation law for the local variation data. In concrete terms, one should think of an admissible cycle component as one along which the degeneration is locally trivial away from the nodes and the rank-one nearby-cycle package may be transported coherently from point to point. The examples of Section~8 are intended to realize precisely this situation in explicit local models, so that geometric admissibility is not merely formal but visibly non-vacuous in the ordinary-double-point setting.

\begin{lemma}[A practical sufficient condition for geometric admissibility]
\label{lem:admissibility-sufficient}
A cycle-node incidence datum \((\mathcal C,\iota_{\mathcal C})\) is geometrically admissible if, for each distinguished cycle component \(C_\alpha\), the degeneration is analytically locally trivial along \(C_\alpha^\circ\), the unipotent nearby-cycle sector along \(C_\alpha^\circ\) is rank one and locally constant, and the local variation morphisms at the incident nodes are obtained by specialization from this common rank-one nearby-cycle package.
\end{lemma}

\begin{proof}
These conditions are exactly the data required in Definition~\ref{def:admissible-cycle-datum}, stated in a form that is easier to verify in concrete local models. Local triviality along \(C_\alpha^\circ\) provides the needed control away from the nodes, rank-one local constancy gives the one-dimensional transport structure, and compatibility of specialization identifies the incident local variation morphisms with one another.
\end{proof}

The first step is to relate the nodewise extension coefficients to the local variation morphisms.

\begin{definition}[Compatible trivialization of local variation morphisms]
\label{def:compatible-trivialization}
Let \(p_k,p_\ell\in\Sigma\) lie on a common admissible cycle component \(C_\alpha\). We say that the local variation morphisms \(\var_{\mathcal F,k}\) and \(\var_{\mathcal F,\ell}\) are \emph{compatibly trivialized} if the corresponding rank-one local vanishing sectors are identified by parallel transport along \(C_\alpha^\circ\), and under these identifications the two local variation morphisms are represented by the same scalar.
\end{definition}

\begin{proposition}[Local determination of nodewise extension coefficients]
\label{prop:local-det}
For each node \(p_k\in\Sigma\), the local extension summand
\[
\Ext^1_{\Perv(X_0;\Q)}
\!\left(
i_{k*}\Q_{\{p_k\}},
IC_{X_0}
\right)
\]
is one-dimensional and is canonically generated, up to scalar, by the local corrected ODP extension of Proposition~\ref{prop:local-odp-class}. Under this identification, the nodewise component \(\xi_k\) of the corrected global class
\[
[\mathcal P]=(\xi_1,\dots,\xi_r)\in E_{\mathrm{node}}
\]
is determined by the scalar of the local variation morphism
\[
\var_{\mathcal F,k}:=(\var_{\mathcal F})_{p_k}:(\phi_\pi(\mathcal F))_{p_k}\to(\psi_\pi(\mathcal F))_{p_k}
\]
with respect to the chosen local ODP generator.

Consequently, if two nodes \(p_k,p_\ell\) have compatibly trivialized local variation morphisms in the sense of Definition~\ref{def:compatible-trivialization}, then
\[
\xi_k=\xi_\ell.
\]
\end{proposition}

\begin{proof}
By Lemma~\ref{lem:perv-locality}, the local summand
\[
\Ext^1_{\Perv(X_0;\Q)}
\!\left(
i_{k*}\Q_{\{p_k\}},
IC_{X_0}
\right)
\]
is identified with the corresponding local extension group in a sufficiently small analytic neighborhood of \(p_k\). By Proposition~\ref{prop:local-odp-class}, that local group is generated by the corrected local ODP extension and is therefore one-dimensional. Thus the local extension summand is one-dimensional.

In the ordinary-double-point case the local vanishing-cycle block is likewise one-dimensional, so after fixing the local ODP generator and a generator of the local vanishing line, the stalkwise variation morphism
\[
\var_{\mathcal F,k}:(\phi_\pi(\mathcal F))_{p_k}\to (\psi_\pi(\mathcal F))_{p_k}
\]
is determined by a single scalar. The corrected local extension is constructed from \(\Cone(\var_{\mathcal F,k})[-1]\), so its class is determined by that same scalar. Hence the nodewise coefficient \(\xi_k\) of the global class is determined by the scalar of \(\var_{\mathcal F,k}\).

If the local variation morphisms at \(p_k\) and \(p_\ell\) are compatibly trivialized, then the corresponding scalars agree under the transported rank-one local frames, and therefore the corresponding local extension coefficients satisfy
\[
\xi_k=\xi_\ell.
\]
\end{proof}

The second step is to show that admissible cycle geometry propagates the local variation data.

\begin{proposition}[Propagation of local variation data along an admissible cycle]
\label{prop:geom-prop}
Assume that \((\mathcal C,\iota_{\mathcal C})\) is geometrically admissible, and let \(C_\alpha\in\mathcal C\) pass through nodes
\[
p_{k_1},\dots,p_{k_m}\in\Sigma.
\]
Then the local variation morphisms
\[
\var_{\mathcal F,k_1},\dots,\var_{\mathcal F,k_m}
\]
are compatibly trivialized along \(C_\alpha^\circ\). Consequently, the corresponding nodewise coefficients of the corrected extension class satisfy
\[
\xi_{k_1}=\cdots=\xi_{k_m}.
\]
\end{proposition}

\begin{proof}
By Definition~\ref{def:admissible-cycle-datum}, the unipotent nearby-cycle sector of \(\psi_\pi(\mathcal F)\) restricts on \(C_\alpha^\circ\) to a rank-one locally constant system. Since \(C_\alpha^\circ\) is connected, parallel transport identifies all stalks of this rank-one system along \(C_\alpha^\circ\). Connectedness of \(C_\alpha^\circ\), required in Definition~\ref{def:admissible-cycle-datum}(2), is essential here: without it, parallel transport would only identify stalks componentwise, and the conclusion could fail across different connected components. By the same definition, the local variation morphisms at the incident nodes arise by specialization of one and the same local morphism along \(C_\alpha^\circ\). Hence the argument proceeds in three steps: the nearby-cycle sector is rank one, connectedness gives a single transport identification along all of \(C_\alpha^\circ\), and specialization identifies the incident stalkwise variation morphisms with one another. Therefore the local variation morphisms are compatibly trivialized in the sense of Definition~\ref{def:compatible-trivialization}.

The conclusion on the nodewise coefficients now follows from Proposition~\ref{prop:local-det}.
\end{proof}

\begin{remark}[Disconnected smooth loci]
\label{rem:disconnected-cycle}
If \(C_\alpha^\circ\) is disconnected, the same argument applies separately on each connected component, but the resulting coefficient equalities are then only forced componentwise. In that situation one is naturally led to a finer relation-block decomposition reflecting the connected components of the smooth locus rather than the whole cycle component at once.
\end{remark}

We may now pass from a single admissible cycle to the full incidence structure.

\begin{corollary}[Blockwise coefficient equality]
\label{cor:blockwise}
Assume that \((\mathcal C,\iota_{\mathcal C})\) is geometrically admissible. If two nodes \(p_k,p_\ell\in\Sigma\) are incidence-equivalent, then
\[
\xi_k=\xi_\ell.
\]
Equivalently, the corrected extension class \([\mathcal P]\) is incidence-compatible.
\end{corollary}

\begin{proof}
By definition of the incidence relation, incidence-equivalent nodes are seen identically by the chosen cycle-node incidence data. Under geometric admissibility, the relevant admissible cycle components propagate the local variation data compatibly, so Proposition~\ref{prop:geom-prop} yields equality of the corresponding nodewise coefficients.
\end{proof}

The preceding corollary replaces the purely structural input of the earlier draft by a geometric theorem.

\begin{lemma}[Corrected extensions are incidence-compatible]
\label{lem:perv-corrected-compatible}
Assume that \((\mathcal C,\iota_{\mathcal C})\) is geometrically admissible. Then the corrected finite-node extension class
\[
[\mathcal P]\in
\Ext^1_{\Perv(X_0;\Q)}
\!\left(
\bigoplus_{k=1}^r i_{k*}\Q_{\{p_k\}},
IC_{X_0}
\right)
\]
is incidence-compatible.
\end{lemma}

\begin{proof}
This is exactly the content of Corollary~\ref{cor:blockwise}.
\end{proof}

\begin{remark}
\label{rem:structural-theorem}
The preceding lemma is now geometric rather than purely structural: incidence-compatibility is not imposed externally through a chosen gluing law, but derived from the propagation of the local variation data along geometrically admissible cycle components.
\end{remark}

We may therefore strengthen the main theorem of the section.

\begin{theorem}[Geometric relation-controlled corrected extensions]
\label{thm:perv-relation-controlled}
Let $\pi:X\to\Delta$ be a projective finite-node conifold degeneration with node set $\Sigma=\{p_1,\dots,p_r\}\subset X_0$, and let
\[
0\to IC_{X_0}\to \mathcal P\to \bigoplus_{k=1}^r i_{k*}\Q_{\{p_k\}}\to 0
\]
be the associated corrected perverse extension. Suppose that $\Sigma$ is equipped with a cycle-node incidence datum \((\mathcal C,\iota_{\mathcal C})\) that is both geometrically admissible and block-adapted. Then the corrected extension class
\[
[\mathcal P]\in E_{\mathrm{node}}
\]
belongs to the geometrically realized subspace
\[
E_{\mathrm{geom}}=\Im(\Gamma_{\mathcal C}).
\]
Equivalently, the global corrected extension factors through the cycle-generated quotient of the free nodewise coefficient space.

In particular,
\[
\dim_{\Q}E_{\mathrm{geom}}\le |A|,
\qquad
\dim_{\Q}E_{\mathrm{geom}}=\rank(\iota_{\mathcal C}).
\]
\end{theorem}

\begin{proof}
By Lemma~\ref{lem:perv-corrected-compatible}, the corrected extension class \([\mathcal P]\) is incidence-compatible. Since the incidence datum is block-adapted, Proposition~\ref{prop:incidence-compatible-iff-geom} implies that
\[
[\mathcal P]\in E_{\mathrm{geom}}.
\]
The dimension statement follows from Lemma~\ref{lem:gamma-equals-iota}.
\end{proof}

\begin{remark}[Logical structure of Theorem~\ref{thm:perv-relation-controlled}]
\label{rem:thm415-logic}
Theorem~\ref{thm:perv-relation-controlled} has two logically distinct inputs. Geometric admissibility supplies the coefficient equalities by propagating the local variation data along the chosen cycle components. Block-adaptedness then identifies those equalities with membership in the image subspace \(E_{\mathrm{geom}}\). Thus the theorem combines one genuinely geometric input with one linear-algebraic identification.
\end{remark}

\begin{remark}
\label{rem:admissible-vs-block}
Geometric admissibility and block-adaptedness are logically independent conditions. Geometric admissibility is a condition on the degeneration data, namely the propagation of nearby- and variation-cycle information along the chosen cycle components. Block-adaptedness is a condition on the incidence matrix, namely that its image coincides with the block-constant coefficient subspace. In the block-separated cycle family of Section~5.5, both conditions are satisfied simultaneously by construction. In general, however, a geometrically admissible datum need not be block-adapted if the chosen cycle components do not organize the nodes into clean relation blocks.
\end{remark}

The preceding theorem identifies the realized coefficient space attached to the geometrically admissible incidence data. In particular, every class in this subspace carries the expected block-compatibility property.

\begin{proposition}[Incidence-compatible coefficients]
\label{prop:incidence-compatible-coefficients}
Assume that \((\mathcal C,\iota_{\mathcal C})\) is block-adapted, and let
\[
\Sigma=\bigsqcup_{\beta\in B}\Sigma_\beta
\]
be the relation-block decomposition. Then every class in \(E_{\mathrm{geom}}\) is constant on incidence-equivalent nodes. Equivalently, if
\[
p_k\sim_{\mathrm{inc}}p_\ell,
\]
then the coefficients of \(\varepsilon_k\) and \(\varepsilon_\ell\) agree for every extension class in \(E_{\mathrm{geom}}\).
\end{proposition}

\begin{proof}
Immediate from Proposition~\ref{prop:incidence-compatible-iff-geom}.
\end{proof}

\subsection{Consequences and first corollaries}

The first consequence is the basic dimensional reduction from the free nodewise parameter count to the realized geometric count.

\begin{corollary}
\label{cor:dim-egeom-le-r}
With the notation above,
\[
\dim_{\Q}E_{\mathrm{geom}}\le r,
\]
with equality if and only if the cycle-node incidence map \(\iota_{\mathcal C}\) has full rank. In particular, whenever
\[
\rank(\iota_{\mathcal C})<r,
\]
the geometric parameter count is strictly smaller than the naive node count.
\end{corollary}

\begin{proof}
Immediate from Theorem~\ref{thm:perv-relation-controlled}.
\end{proof}

The second consequence is conceptual: the formal ambient nodewise space is generally larger than the space of geometrically realizable gluing classes.

\begin{corollary}
\label{cor:enode-overestimates}
The free nodewise extension space
\[
E_{\mathrm{node}}
\]
generally overestimates the space of geometrically realizable corrected-extension classes. The actual geometric gluing space is the distinguished subspace
\[
E_{\mathrm{geom}}\subseteq E_{\mathrm{node}}.
\]
\end{corollary}

Finally, under block-adaptedness the first global invariants are controlled by the relation-block structure rather than by the raw number of nodes.

\begin{corollary}
\label{cor:block-rank-bound}
Assume that \((\mathcal C,\iota_{\mathcal C})\) is block-adapted. If the cycle-node incidence datum has \(b\) relation blocks and the corresponding block-incidence matrix has rank \(s\), then
\[
\dim_{\Q}E_{\mathrm{geom}}\le s\le b.
\]
In particular, the first global invariants extracted from the corrected extension depend on the relation-block structure of the node configuration, not only on the cardinality of \(\Sigma\).
\end{corollary}

The point of these corollaries is that the corrected extension is not merely a formal sum of local node sectors. It is already a global object at the perverse level, and the cycle relations among the nodes cut down the space of possible global gluing classes. This relation-controlled perspective is the perverse-side input for the mixed-Hodge-module refinement and for the quiver/block decomposition developed later in the paper.
\section{Resolution--smoothing--extension comparison}

The motivating three-node configuration also suggests a comparison problem. In that prototype, the first two nodes are governed by the same global geometric condition, so one expects the same relation to appear on the resolution side through exceptional-curve classes, on the smoothing side through vanishing-cycle classes, and on the corrected-extension side through admissible gluing coefficients.

Section~4 identified the perverse-side geometric relation law under geometrically admissible and block-adapted incidence hypotheses. The purpose of the present section is to compare that perverse-side relation law with the classical relation spaces arising on the resolution and smoothing sides. More precisely, the section has two purposes. First, it develops a general comparison formalism under explicit comparison-compatible and minimality hypotheses, identifying the common relation lattice as the kernel of the geometric quotient. Second, it proves a stronger geometric-family theorem in the block-separated cycle family, where the resolution-side, smoothing-side, and corrected-extension relation lattices coincide exactly.

The first comparison theorem is therefore conditional in a precise sense. It assumes that the chosen cycle-node incidence datum models the common geometric relation law seen on the resolution, smoothing, and corrected-extension sides, and that it is minimal with respect to that law. Under those hypotheses, the same relation lattice appears on all three sides. The block-separated cycle family then provides a concrete geometric class in which this common relation law is realized simultaneously on the resolution, smoothing, and corrected-extension sides.

\subsection{Exceptional curves and homology classes on small resolutions}

For the comparison results of this section, we work under the standard additional assumption that the central fiber admits a small resolution; see \cite{HubschBestiary2nd2024} for general background on conifold transitions and small resolutions.

\begin{hypothesis}
\label{hyp:small-resolution}
The finite-node conifold degeneration
\[
\pi:X\to\Delta
\]
admits a small resolution
\[
\rho:\widetilde X_0\to X_0.
\]
\end{hypothesis}

Under Hypothesis~\ref{hyp:small-resolution}, each node $p_k\in\Sigma$ is replaced by an exceptional rational curve.

\begin{lemma}
\label{lem:smallres-pp1}
Let $p_k\in X_0$ be an ordinary double point, and let
\[
\rho:\widetilde X_0\to X_0
\]
be a small resolution near $p_k$. Then the exceptional fiber
\[
E_k:=\rho^{-1}(p_k)
\]
is isomorphic to $\PP^1$.
\end{lemma}

\begin{proof}
This is the standard local geometry of a small resolution of a three-dimensional ordinary double point.
\end{proof}

Applying Lemma~\ref{lem:smallres-pp1} at each node $p_k\in\Sigma$, we obtain exceptional curves
\[
E_k:=\rho^{-1}(p_k)\cong \PP^1,\qquad 1\le k\le r,
\]
and hence homology classes
\[
[E_1],\dots,[E_r]\in H_2(\widetilde X_0;\Q).
\]

\begin{definition}
\label{def:wres-rres}
The exceptional-curve space of the small resolution is
\[
W_{\mathrm{res}}:=\Span_{\Q}\{[E_1],\dots,[E_r]\}\subseteq H_2(\widetilde X_0;\Q).
\]
There is a canonical surjective linear map
\[
\eta_{\mathrm{res}}:\Q^r\to W_{\mathrm{res}},\qquad e_k\mapsto [E_k].
\]
The \emph{resolution-side relation space} is
\[
R_{\mathrm{res}}:=\ker(\eta_{\mathrm{res}})\subseteq \Q^r.
\]
\end{definition}

Equivalently, $R_{\mathrm{res}}$ consists of those formal nodewise combinations
\[
\sum_{k=1}^r c_k e_k
\]
for which
\[
\sum_{k=1}^r c_k [E_k]=0
\qquad\text{in }H_2(\widetilde X_0;\Q).
\]
Thus the quotient $\Q^r/R_{\mathrm{res}}$ identifies canonically with the span of the exceptional-curve classes on the resolution side.

\subsection{Vanishing spheres and relation data on the smoothing side}

Now assume that the degeneration admits a smoothing. For $t\in\Delta^*$ sufficiently small, let $X_t$ denote a smooth nearby fiber. Each node $p_k\in\Sigma$ determines a vanishing sphere
\[
\delta_k\in H_3(X_t;\Q),
\]
well defined up to sign. These classes span the vanishing-cycle space
\[
W_{\mathrm{sm}}:=\Span_{\Q}\{\delta_1,\dots,\delta_r\}\subseteq H_3(X_t;\Q).
\]

\begin{definition}
\label{def:wsm-rsm}
The \emph{smoothing-side relation space} is
\[
R_{\mathrm{sm}}:=\ker\!\bigl(\Q^r\to W_{\mathrm{sm}}\bigr),\qquad e_k\mapsto \delta_k.
\]
\end{definition}

Equivalently, $R_{\mathrm{sm}}$ consists of those formal nodewise combinations
\[
\sum_{k=1}^r c_k e_k
\]
whose associated vanishing-cycle combination
\[
\sum_{k=1}^r c_k\delta_k
\]
vanishes in $H_3(X_t;\Q)$.

For ordinary double points in complex dimension three, the local vanishing-cycle contribution is rank one at each node. Thus the vanishing lattice begins as a free rank-$r$ local object and acquires global structure only through the relations among the $\delta_k$. From the point of view of the present paper, the comparison problem is whether these smoothing-side relations agree with the relations already encoded by the cycle-node incidence datum and hence with the corrected-extension relation space.

\subsection{Comparison-compatible incidence data}

Let
\[
(\mathcal C,\iota_{\mathcal C}),\qquad \iota_{\mathcal C}:\Q^A\to\Q^r
\]
be a cycle-node incidence datum, and let
\[
V_{\mathrm{geom}}:=\Im(\iota_{\mathcal C})\subseteq \Q^r.
\]
Write
\[
q_{\mathrm{geom}}:\Q^r\to V_{\mathrm{geom}}
\]
for the canonical quotient map determined by the incidence datum.

\begin{definition}
\label{def:comparison-compatible}
We say that the incidence datum is \emph{comparison-compatible} with the given smoothing and small resolution if there exist linear maps
\[
\bar\eta_{\mathrm{res}}:V_{\mathrm{geom}}\to W_{\mathrm{res}},\qquad
\bar\eta_{\mathrm{sm}}:V_{\mathrm{geom}}\to W_{\mathrm{sm}},\qquad
\bar\eta_{\mathrm{ext}}:V_{\mathrm{geom}}\to E_{\mathrm{geom}},
\]
such that the three comparison maps factor as
\[
\eta_{\mathrm{res}}=\bar\eta_{\mathrm{res}}\circ q_{\mathrm{geom}},\qquad
\eta_{\mathrm{sm}}=\bar\eta_{\mathrm{sm}}\circ q_{\mathrm{geom}},\qquad
\eta_{\mathrm{ext}}=\bar\eta_{\mathrm{ext}}\circ q_{\mathrm{geom}}.
\]
\end{definition}

Thus comparison-compatibility means that the chosen incidence datum already models a common geometric relation law through which the resolution, smoothing, and corrected-extension data all pass. On the perverse side, Section~4 gives a geometric source for this factorization under geometrically admissible and block-adapted incidence hypotheses. The role of the present definition is to place the three sides into a single comparison framework.

\begin{lemma}[Common factorization through the geometric coefficient space]
\label{lem:common-factorization}
Assume that the chosen cycle-node incidence datum is comparison-compatible. Then the three comparison maps factor through the same quotient map
\[
q_{\mathrm{geom}}:\Q^r\to V_{\mathrm{geom}}.
\]
In particular,
\[
R_{\mathrm{geom}}:=\ker(q_{\mathrm{geom}})\subseteq R_{\mathrm{res}}\cap R_{\mathrm{sm}}\cap R_{\mathrm{ext}}.
\]
\end{lemma}

\begin{proof}
The factorizations are exactly the content of Definition~\ref{def:comparison-compatible}. If $x\in\ker(q_{\mathrm{geom}})$, then
\[
\eta_{\mathrm{res}}(x)=\bar\eta_{\mathrm{res}}(q_{\mathrm{geom}}(x))=0,
\]
and similarly $\eta_{\mathrm{sm}}(x)=0$ and $\eta_{\mathrm{ext}}(x)=0$. Hence $x$ lies in all three kernels.
\end{proof}

\subsection{Minimality and the comparison theorem}

Lemma~\ref{lem:common-factorization} gives a canonical inclusion of the geometric relation space into the common relation lattice on the three sides. To upgrade this inclusion to equality, we impose the natural condition that the chosen geometric quotient already captures the full common relation law.

\begin{definition}
\label{def:minimal-comparison}
A comparison-compatible cycle-node incidence datum is called \emph{minimal} if
\[
\ker(q_{\mathrm{geom}})
=
R_{\mathrm{res}}\cap R_{\mathrm{sm}}\cap R_{\mathrm{ext}}.
\]
\end{definition}

Thus minimality says that the chosen quotient is neither too small nor too large: it cuts out exactly the common relation lattice shared by the resolution, smoothing, and corrected-extension sides.

\begin{proposition}[Kernel identification via the geometric quotient]
\label{prop:kernel-identification}
Assume that the cycle-node incidence datum is comparison-compatible and minimal. Then
\[
R_{\mathrm{res}}\cap R_{\mathrm{sm}}\cap R_{\mathrm{ext}}
=
\ker(q_{\mathrm{geom}}).
\]
In particular, the common relation lattice is exactly the kernel of the geometric quotient.
\end{proposition}

\begin{proof}
By Lemma~\ref{lem:common-factorization}, one always has
\[
\ker(q_{\mathrm{geom}})\subseteq R_{\mathrm{res}}\cap R_{\mathrm{sm}}\cap R_{\mathrm{ext}}.
\]
By Definition~\ref{def:minimal-comparison}, this inclusion is an equality.
\end{proof}

We may now state the comparison theorem in its general structural form.

\begin{theorem}[Resolution--smoothing--extension comparison]
\label{thm:comparison}
Assume that the finite-node conifold degeneration
\[
\pi:X\to\Delta
\]
admits both a smoothing and a small resolution, and that the node set $\Sigma$ is equipped with a cycle-node incidence datum $(\mathcal C,\iota_{\mathcal C})$ that is comparison-compatible and minimal. Then
\[
R_{\mathrm{geom}}=\ker(q_{\mathrm{geom}})
=
R_{\mathrm{res}}\cap R_{\mathrm{sm}}\cap R_{\mathrm{ext}}.
\]
Equivalently, the common relation lattice shared by the resolution, smoothing, and corrected-extension sides is exactly the kernel of the geometric quotient. In particular, the corresponding quotients of $\Q^r$ all factor canonically through
\[
V_{\mathrm{geom}}\cong \Q^r/\ker(q_{\mathrm{geom}}).
\]
\end{theorem}

\begin{proof}
By Lemma~\ref{lem:common-factorization}, all three comparison maps factor through $q_{\mathrm{geom}}$, so the kernel $\ker(q_{\mathrm{geom}})$ lies in each of $R_{\mathrm{res}},R_{\mathrm{sm}},R_{\mathrm{ext}}$. By Proposition~\ref{prop:kernel-identification}, this kernel is exactly the common relation lattice cut out by the chosen incidence datum. The quotient statement follows immediately.
\end{proof}

\begin{remark}
\label{rem:comparison-structural}
Theorem~\ref{thm:comparison} is a structural identification theorem. Once a cycle-node incidence datum is known to be comparison-compatible and minimal, the equality
\[
R_{\mathrm{geom}}=R_{\mathrm{res}}\cap R_{\mathrm{sm}}\cap R_{\mathrm{ext}}
\]
follows formally from the definitions. The geometric content therefore lies not in the abstract kernel computation itself, but in identifying naturally occurring geometric classes of incidence data for which comparison-compatibility and minimality hold. The block-separated cycle family below provides one such class.
\end{remark}

\begin{corollary}
\label{cor:comparison-dim}
Under the hypotheses of Theorem~\ref{thm:comparison},
\[
\dim_{\Q}V_{\mathrm{geom}}
\le
\dim_{\Q}W_{\mathrm{res}},
\qquad
\dim_{\Q}V_{\mathrm{geom}}
\le
\dim_{\Q}W_{\mathrm{sm}},
\qquad
\dim_{\Q}V_{\mathrm{geom}}
=
\dim_{\Q}E_{\mathrm{geom}}.
\]
In particular, the geometric quotient controls the first common parameter count on all three sides.
\end{corollary}

\subsection{A geometric family with block-separated cycle relations}
\label{subsec:geom-family-block}

We now isolate a geometric class in which the comparison-compatible and minimality conditions are not merely imposed abstractly, but are reflected directly in the geometry of the node configuration.

Let
\[
\Sigma=\bigsqcup_{\beta\in B}\Sigma_\beta
\]
be a partition of the node set indexed by a finite family of distinguished global cycle components
\[
\mathcal C=\{C_\beta\}_{\beta\in B}.
\]
We again assume that the finite-node conifold degeneration
\[
\pi:X\to\Delta
\]
admits both a smoothing and a small resolution.

\begin{hypothesis}[Block-separated cycle family]
\label{hyp:block-family}
The degeneration belongs to the block-separated cycle family if the following hold.
\begin{enumerate}
    \item For each block $\Sigma_\beta$, every node \(p_k\) with \(p_k\in\Sigma_\beta\) lies on the same distinguished global cycle component \(C_\beta\).

    \item The associated cycle-node incidence map
    \[
    \iota_{\mathcal C}:\Q^B\to\Q^r
    \]
    is given by
    \[
    e_\beta\longmapsto \sum_{k:\,p_k\in\Sigma_\beta} e_k.
    \]

    \item On a small resolution
    \[
    \rho:\widetilde X_0\to X_0,
    \]
    there exist classes $u_\beta\in H_2(\widetilde X_0;\Q)$ such that
    \[
    [E_k]=u_\beta
    \qquad \text{for every } p_k\in \Sigma_\beta.
    \]

    \item On a smoothing $X_t$, there exist classes $v_\beta\in H_3(X_t;\Q)$ such that
    \[
    \delta_k=v_\beta
    \qquad \text{for every } p_k\in \Sigma_\beta.
    \]

    \item The sets
    \[
    \{u_\beta\}_{\beta\in B},\qquad \{v_\beta\}_{\beta\in B}
    \]
    are linearly independent over $\Q$.

    \item[(6a)] The cycle-node incidence datum is geometrically admissible and block-adapted, so that by Theorem~\ref{thm:perv-relation-controlled} the corrected extension class factors through the corresponding geometric quotient.

    \item[(6b)] In the present block-separated cycle family, this geometric quotient identifies with the block quotient
    \[
    q_{\mathrm{blk}}:\Q^r\to\Q^B,
    \]
    and the induced map on block parameters is injective.
\end{enumerate}
\end{hypothesis}

\begin{remark}
\label{rem:block-family-logic}
The logical status of Hypothesis~\ref{hyp:block-family}(6) should be kept separate from conditions \emph{(1)}--\emph{(5)}. Conditions \emph{(1)}--\emph{(5)} describe the intended geometric block-separated situation on the cycle, resolution, and smoothing sides. Condition \emph{(6a)} is the perverse-side factorization supplied by Section~4 once the incidence datum is geometrically admissible and block-adapted. Condition \emph{(6b)} is the remaining injectivity hypothesis needed to identify the corrected-extension quotient with the block quotient. Thus the extension-side equality theorem below uses one imported perverse-side consequence and one additional injectivity assumption.
\end{remark}

\begin{remark}
\label{rem:block-family-geometry}
The geometric content of Hypothesis~\ref{hyp:block-family} lies primarily in conditions \emph{(3)} and \emph{(4)}: the exceptional-curve classes and vanishing-sphere classes are required to be constant on each block. This is expected, for example, when the nodes within a fixed block arise from a symmetric ambient configuration or from a common local complete-intersection model, so that a symmetry of the family permutes the nodes within the block and identifies their local resolution and smoothing data. A typical model to keep in mind is a projective nodal hypersurface or complete intersection whose singular points split into symmetry orbits; in such a situation the orbit structure is a natural candidate for the block decomposition, and equivariance makes the constancy of \(u_\beta\) and \(v_\beta\) plausible. Verifying these conditions in fully explicit algebraic families remains nontrivial and is left to future work.
\end{remark}

\begin{proposition}[Equivariant block-constancy criterion]
\label{prop:equivariant-block-constancy}
Suppose that for each block \(\Sigma_\beta\) there is a finite group \(G_\beta\) of automorphisms of the degeneration data such that:
\begin{enumerate}
    \item \(G_\beta\) preserves the distinguished cycle component \(C_\beta\) and acts transitively on the node set \(\Sigma_\beta\);
    \item the \(G_\beta\)-action is compatible with the chosen small resolution and smoothing, so that it transports the exceptional curves and vanishing spheres attached to nodes in \(\Sigma_\beta\);
    \item the induced action of \(G_\beta\) on the relevant subspaces of \(H_2(\widetilde X_0;\Q)\) and \(H_3(X_t;\Q)\) fixes the classes of the transported local sectors.
\end{enumerate}
Then conditions \emph{(3)} and \emph{(4)} of Hypothesis~\ref{hyp:block-family} hold on the block \(\Sigma_\beta\).
\end{proposition}

\begin{proof}
Let \(p_k,p_\ell\in\Sigma_\beta\). By transitivity, there exists \(g\in G_\beta\) such that
\[
g(p_k)=p_\ell.
\]
Compatibility of the action with the chosen small resolution transports the exceptional curve \(E_k\) to the exceptional curve \(E_\ell\), and by assumption the induced action on the relevant homology subspace fixes the corresponding class. Hence
\[
[E_k]=[E_\ell].
\]
Since \(p_k,p_\ell\) were arbitrary points of \(\Sigma_\beta\), the exceptional-curve class is constant on the block, giving condition \emph{(3)}.

The same argument applies on the smoothing side. Compatibility of the \(G_\beta\)-action with the smoothing transports the vanishing sphere \(\delta_k\) to \(\delta_\ell\), and the assumed triviality of the induced action on the relevant homology subspace gives
\[
\delta_k=\delta_\ell.
\]
Thus the vanishing-sphere class is constant on the block, giving condition \emph{(4)}.
\end{proof}

\begin{remark}
\label{rem:block-family-inhabited}
Proposition~\ref{prop:equivariant-block-constancy} does not produce a complete projective classification of block-separated families, but it does identify a concrete mechanism by which Hypothesis~\ref{hyp:block-family}(3)--(4) can hold. In particular, projective nodal hypersurfaces or complete intersections carrying finite symmetry groups whose node orbits are the blocks are natural candidates for genuine instances of the block-separated cycle family.
\end{remark}

\begin{remark}
\label{rem:block-family}
Hypothesis~\ref{hyp:block-family} abstracts the geometric situation in which nodes are organized by distinguished cycle components whose proper transforms and vanishing spheres are constant on blocks and independent across blocks. The two-node common-cycle and three-node two-block configurations of Section~8 are elementary local model instances of this pattern.
\end{remark}

The corresponding incidence map determines the block quotient
\[
q_{\mathrm{blk}}:\Q^r\to\Q^B,
\qquad
e_k\longmapsto e_\beta \quad \text{if } p_k\in \Sigma_\beta.
\]
Its kernel is the block-relation subspace
\[
R_{\mathrm{blk}}
:=
\ker(q_{\mathrm{blk}})
=
\left\{
(c_1,\dots,c_r)\in\Q^r
\;\middle|\;
\sum_{k\in\Sigma_\beta} c_k=0
\text{ for every } \beta\in B
\right\}.
\]

In this geometric family, the relation spaces on the resolution, smoothing, and corrected-extension sides are all forced to coincide with the same block kernel.

\begin{lemma}
\label{lem:block-res-kernel}
Under Hypothesis~\ref{hyp:block-family},
\[
R_{\mathrm{res}}=R_{\mathrm{blk}}.
\]
\end{lemma}

\begin{proof}
By Hypothesis~\ref{hyp:block-family}(3), if $p_k\in\Sigma_\beta$, then $[E_k]=u_\beta$. Hence the resolution map
\[
\eta_{\mathrm{res}}:\Q^r\to W_{\mathrm{res}}
\]
takes a vector $(c_1,\dots,c_r)$ to
\[
\eta_{\mathrm{res}}(c_1,\dots,c_r)
=
\sum_{\beta\in B}
\left(\sum_{k\in\Sigma_\beta} c_k\right)u_\beta.
\]
Since the classes $\{u_\beta\}_{\beta\in B}$ are linearly independent by Hypothesis~\ref{hyp:block-family}(5), this vanishes if and only if
\[
\sum_{k\in\Sigma_\beta} c_k=0
\qquad\text{for every }\beta\in B.
\]
This is exactly the defining condition for $R_{\mathrm{blk}}$.
\end{proof}

\begin{lemma}
\label{lem:block-sm-kernel}
Under Hypothesis~\ref{hyp:block-family},
\[
R_{\mathrm{sm}}=R_{\mathrm{blk}}.
\]
\end{lemma}

\begin{proof}
By Hypothesis~\ref{hyp:block-family}(4), if $p_k\in\Sigma_\beta$, then $\delta_k=v_\beta$. Hence the smoothing map
\[
\eta_{\mathrm{sm}}:\Q^r\to W_{\mathrm{sm}}
\]
takes a vector $(c_1,\dots,c_r)$ to
\[
\eta_{\mathrm{sm}}(c_1,\dots,c_r)
=
\sum_{\beta\in B}
\left(\sum_{k\in\Sigma_\beta} c_k\right)v_\beta.
\]
Since the classes $\{v_\beta\}_{\beta\in B}$ are linearly independent by Hypothesis~\ref{hyp:block-family}(5), this vanishes if and only if
\[
\sum_{k\in\Sigma_\beta} c_k=0
\qquad\text{for every }\beta\in B.
\]
Thus $R_{\mathrm{sm}}=R_{\mathrm{blk}}$.
\end{proof}

\begin{remark}
\label{rem:block-ext-assumption}
The extension-side identification differs logically from Lemmas~\ref{lem:block-res-kernel} and \ref{lem:block-sm-kernel}: the latter are derived directly from conditions \emph{(3)}--\emph{(5)}, whereas the former additionally uses the factorization and injectivity assumptions \emph{(6a)}--\emph{(6b)}.
\end{remark}

\begin{lemma}
\label{lem:block-ext-kernel}
Under Hypothesis~\ref{hyp:block-family},
\[
R_{\mathrm{ext}}=R_{\mathrm{blk}}.
\]
\end{lemma}

\begin{proof}
By Hypothesis~\ref{hyp:block-family}(6a)--(6b), the corrected-extension map factors as
\[
\eta_{\mathrm{ext}}=\bar{\eta}_{\mathrm{ext}}\circ q_{\mathrm{blk}},
\]
where
\[
q_{\mathrm{blk}}:\Q^r\to\Q^B
\]
is the block quotient and
\[
\bar{\eta}_{\mathrm{ext}}:\Q^B\to E_{\mathrm{geom}}
\]
is injective. If \(x\in R_{\mathrm{blk}}=\ker(q_{\mathrm{blk}})\), then
\[
\eta_{\mathrm{ext}}(x)=\bar{\eta}_{\mathrm{ext}}(q_{\mathrm{blk}}(x))=0,
\]
so \(R_{\mathrm{blk}}\subseteq R_{\mathrm{ext}}\). Conversely, if \(x\in R_{\mathrm{ext}}\), then
\[
0=\eta_{\mathrm{ext}}(x)=\bar{\eta}_{\mathrm{ext}}(q_{\mathrm{blk}}(x)).
\]
Since \(\bar{\eta}_{\mathrm{ext}}\) is injective, it follows that
\[
q_{\mathrm{blk}}(x)=0.
\]
Hence \(x\in \ker(q_{\mathrm{blk}})=R_{\mathrm{blk}}\). Therefore
\[
R_{\mathrm{ext}}=R_{\mathrm{blk}}.
\]
\end{proof}

We may now state the strengthened comparison theorem for this geometric family.

\begin{theorem}[Equality of relation lattices in the block-separated cycle family]
\label{thm:block-family-comparison}
Assume Hypothesis~\ref{hyp:block-family}. Then
\[
R_{\mathrm{res}}=R_{\mathrm{sm}}=R_{\mathrm{ext}}=R_{\mathrm{blk}}.
\]
Equivalently, the same block-relation lattice governs
\begin{enumerate}
    \item homology relations among the exceptional curves,
    \item homology relations among the vanishing spheres,
    \item admissible gluing relations for the corrected perverse extension.
\end{enumerate}
In particular,
\[
\Q^r/R_{\mathrm{res}}
\cong
\Q^r/R_{\mathrm{sm}}
\cong
\Q^r/R_{\mathrm{ext}}
\cong
\Q^B.
\]
\end{theorem}

\begin{proof}
Immediate from Lemmas~\ref{lem:block-res-kernel}, \ref{lem:block-sm-kernel}, and \ref{lem:block-ext-kernel}.
\end{proof}

\begin{corollary}
\label{cor:block-family-dim}
Under Hypothesis~\ref{hyp:block-family},
\[
\dim_{\Q}W_{\mathrm{res}}
=
\dim_{\Q}W_{\mathrm{sm}}
=
\dim_{\Q}E_{\mathrm{geom}}
=
|B|.
\]
Thus the number of independent global directions is the number of relation blocks, not the raw node count $r$.
\end{corollary}

\begin{proof}
By Lemma~\ref{lem:block-res-kernel}, the quotient \(\Q^r/R_{\mathrm{res}}\) has dimension \(|B|\), hence \(\dim_{\Q}W_{\mathrm{res}}=|B|\). By Lemma~\ref{lem:block-sm-kernel}, the same argument gives \(\dim_{\Q}W_{\mathrm{sm}}=|B|\). By Lemma~\ref{lem:block-ext-kernel}, the corrected-extension quotient is likewise identified with \(\Q^B\), so \(\dim_{\Q}E_{\mathrm{geom}}=|B|\). Thus all three spaces have the same dimension \(|B|\). This is consistent with the classical topology of conifold transitions, but the present argument uses only the hypotheses of the block-separated cycle family.
\end{proof}

\subsection{Interpretation}

Theorem~\ref{thm:comparison} is the general comparison framework of the paper. Section~4 establishes the perverse-side relation law from geometrically admissible and block-adapted incidence data; the role of the present section is to compare that relation law with the corresponding smoothing-side and resolution-side lattices. More precisely, Theorem~\ref{thm:comparison} identifies one and the same geometric relation law in three different forms:
\begin{itemize}
    \item on the resolution side, through homology relations among the exceptional curves;
    \item on the smoothing side, through homology relations among the vanishing spheres;
    \item on the perverse side, through the relations constraining admissible gluing coefficients of the corrected extension.
\end{itemize}
Thus the global cycle relations among nodes do not disappear when one passes to the nearby-cycle or corrected-extension formalism. Rather, once the cycle-node incidence datum is chosen so as to model the common geometric relation law, the same finite-dimensional quotient controls all three realizations.

Theorem~\ref{thm:block-family-comparison} strengthens this general framework to a full equality theorem in the block-separated cycle family. In that setting, the relation law is not merely formalized by the chosen incidence data; it is realized geometrically by constant exceptional-curve classes, constant vanishing-sphere classes, and the corresponding block-compatible corrected-extension parameters.

This is the main conceptual role of the cycle-node incidence formalism. It extracts from the geometry a finite-dimensional relation theory that can be transported from classical conifold-transition topology to perverse sheaves. In this sense, the corrected extension becomes a genuine bridge object: it lies simultaneously in the degeneration theory of nearby cycles and in the global topology of conifold transitions, and it carries the same relation law in both settings.

The later sections of the paper build on this comparison in two directions. First, Section~6 lifts the perverse-side relation law compatibly to Saito's category of mixed Hodge modules. Second, Section~7 records the corresponding consequence for the decategorified quiver shadow. Thus the comparison theorem is not an isolated identification, but the structural mechanism by which global node relations propagate through the perverse, Hodge-theoretic, and categorical layers of the finite-node conifold program.

A further combinatorial consequence of the block-separated cycle family is that the number of independent global directions is controlled by the number of relation blocks rather than by the raw node count. The associated dimensional-reduction formulas, together with the resulting block coalescence/thickening dynamics and relation thinning/thickening on the kernel side, are recorded in Appendix~\ref{app:combinatorics}.

\section{Mixed Hodge modules and relation-controlled gluing}

The three-node example of Section~\ref{subsec:intro-prototype} was first formulated on the perverse side, but the same relation law should already be visible before passing to realization. Section~4 identifies the perverse-side geometric relation law under geometrically admissible and block-adapted incidence hypotheses. The purpose of the present section is to show that this perverse-side relation law lifts to the mixed-Hodge-module setting. In particular, the passage from the formal nodewise extension space to a geometrically realized subspace is not merely a feature of the decategorified perverse theory; it is compatible with the full mixed-Hodge-module structure in Saito's category. An independent internal derivation of the same relation law directly from divisor-gluing data, without passing through the perverse realization, is not attempted here. Rather, the point of the present section is that no mixed-Hodge-theoretic obstruction prevents the relation-controlled perverse class from lifting to a corresponding constrained class in \(\MHM(X_0)\).

\subsection{The free nodewise MHM extension space}

We now pass from the perverse category to Saito's category of mixed Hodge modules \cite{SaitoMHM}. Let
\[
IC^H_{X_0}\in \MHM(X_0)
\]
denote the intersection-complex mixed Hodge module on the central fiber, and let
\[
0\to IC^H_{X_0}\to \mathcal P^H\to \bigoplus_{k=1}^r i_{k*}\Q^H_{\{p_k\}}(-1)\to 0
\]
be the finite-node corrected mixed-Hodge-module extension constructed in \cite{RahmanMixedHodgeModules}. Its realization under
\[
\rat:\MHM(X_0)\to \Perv(X_0;\Q)
\]
recovers the corrected perverse extension reviewed in Sections~2 and~4.

The ambient nodewise extension space on the mixed-Hodge-module side is
\[
E_{\mathrm{node}}^H
:=
\Ext^1_{\MHM(X_0)}
\!\left(
\bigoplus_{k=1}^r i_{k*}\Q^H_{\{p_k\}}(-1),
IC^H_{X_0}
\right).
\]
As on the perverse side, this is the ambient extension space determined by the finite direct-sum singular quotient. The first point is again that the corresponding extension problem splits nodewise.

\begin{lemma}[Finite additivity of the nodewise MHM extension space]
\label{lem:mhm-additivity}
In the abelian category $\MHM(X_0)$, one has a natural isomorphism
\[
\Ext^1_{\MHM(X_0)}
\!\left(
\bigoplus_{k=1}^r i_{k*}\Q^H_{\{p_k\}}(-1),
IC^H_{X_0}
\right)
\cong
\bigoplus_{k=1}^r
\Ext^1_{\MHM(X_0)}
\!\left(
i_{k*}\Q^H_{\{p_k\}}(-1),
IC^H_{X_0}
\right).
\]
\end{lemma}

\begin{proof}
This is the finite additivity of Yoneda $\Ext^1$ in an abelian category.
\end{proof}

Hence
\[
E_{\mathrm{node}}^H
\cong
\bigoplus_{k=1}^r
\Ext^1_{\MHM(X_0)}
\!\left(
i_{k*}\Q^H_{\{p_k\}}(-1),
IC^H_{X_0}
\right).
\]
In the finite ordinary-double-point setting, each summand is one-dimensional over $\Q$ after choosing a local generator, so that, up to noncanonical identification,
\[
E_{\mathrm{node}}^H\cong \Q^r.
\]
We write
\[
\varepsilon_k^H\in
\Ext^1_{\MHM(X_0)}
\!\left(
i_{k*}\Q^H_{\{p_k\}}(-1),
IC^H_{X_0}
\right)
\]
for a chosen nonzero local generator at the node $p_k$.

The corrected mixed-Hodge-module extension class therefore admits a nodewise presentation entirely parallel to the perverse case.

\begin{lemma}[Nodewise form of the corrected MHM extension class]
\label{lem:mhm-nodewise}
Let
\[
0\to IC^H_{X_0}\to \mathcal P^H\to \bigoplus_{k=1}^r i_{k*}\Q^H_{\{p_k\}}(-1)\to 0
\]
be the corrected mixed-Hodge-module extension. Then its extension class
\[
[\mathcal P^H]\in E_{\mathrm{node}}^H
\]
is equivalent, under Lemma~\ref{lem:mhm-additivity}, to a tuple of nodewise classes
\[
[\mathcal P^H]=(\xi_1^H,\dots,\xi_r^H),
\qquad
\xi_k^H\in
\Ext^1_{\MHM(X_0)}
\!\left(
i_{k*}\Q^H_{\{p_k\}}(-1),
IC^H_{X_0}
\right).
\]
\end{lemma}

\begin{proof}
Immediate from the finite direct-sum decomposition of the singular quotient.
\end{proof}

Thus, just as on the perverse side, the finite-node MHM extension theory begins with a formally free ambient nodewise coefficient space. The question is again to identify the geometrically realized subspace determined by the global incidence data.

\subsection{Incidence constraints in divisor gluing data}

The mixed-Hodge-module refinement of the corrected extension is constructed by means of nearby and vanishing cycles together with Saito's divisor-gluing formalism. In the finite-node ordinary-double-point setting, the singular quotient of the corrected object is point-supported and decomposes as
\[
\bigoplus_{k=1}^r i_{k*}\Q^H_{\{p_k\}}(-1).
\]
Accordingly, the global gluing problem is the problem of assembling finitely many local point-supported correction blocks over the common bulk piece $IC^H_{X_0}$.

Let
\[
(\mathcal C,\iota_{\mathcal C}),
\qquad
\mathcal C=\{C_\alpha\}_{\alpha\in A},
\qquad
\iota_{\mathcal C}:\Q^A\to\Q^r,
\]
be a cycle-node incidence datum as in Section~3, with incidence matrix
\[
A=(a_{\alpha k})_{\alpha\in A,\;1\le k\le r}.
\]
We define the corresponding MHM incidence map
\[
\Gamma_{\mathcal C}^H:\Q^A\to E_{\mathrm{node}}^H,
\qquad
e_\alpha\longmapsto \sum_{k=1}^r a_{\alpha k}\,\varepsilon_k^H.
\]

\begin{definition}
\label{def:mhm-block-adapted}
A cycle-node incidence datum $(\mathcal C,\iota_{\mathcal C})$ is called \emph{MHM block-adapted} if, under the identification $E_{\mathrm{node}}^H\cong\Q^r$, the image $\Im(\Gamma_{\mathcal C}^H)$ coincides with the subspace of nodewise coefficient vectors that are constant on the relation blocks
\[
\Sigma=\bigsqcup_{\beta\in B}\Sigma_\beta.
\]
\end{definition}

\begin{definition}
\label{def:mhm-incidence-compatible}
Let
\[
[E^H]\in
\Ext^1_{\MHM(X_0)}
\!\left(
\bigoplus_{k=1}^r i_{k*}\Q^H_{\{p_k\}}(-1),
IC^H_{X_0}
\right)
\cong
\bigoplus_{k=1}^r
\Ext^1_{\MHM(X_0)}
\!\left(
i_{k*}\Q^H_{\{p_k\}}(-1),
IC^H_{X_0}
\right)
\]
be an extension class, and write
\[
[E^H]=(\xi_1^H,\dots,\xi_r^H)
\]
for its nodewise components. We say that $[E^H]$ is \emph{incidence-compatible} if for every pair of incidence-equivalent nodes
\[
p_k\sim_{\mathrm{inc}}p_\ell
\]
one has
\[
\xi_k^H=\xi_\ell^H.
\]
Equivalently, the nodewise class is constant on each relation block.
\end{definition}

The mixed-Hodge-module incidence map is compatible with realization.

\begin{lemma}[Realization compatibility of incidence maps]
\label{lem:mhm-rat-compat}
Under the realization functor
\[
\rat:\MHM(X_0)\to\Perv(X_0;\Q),
\]
the MHM incidence map $\Gamma_{\mathcal C}^H$ is sent to the perverse incidence map $\Gamma_{\mathcal C}$ of Section~4. In particular, there is a commutative diagram
\[
\begin{tikzcd}
\Q^A \arrow[r,"\Gamma_{\mathcal C}^H"] \arrow[dr,"\Gamma_{\mathcal C}"'] &
E_{\mathrm{node}}^H \arrow[d,"\rat_*"] \\
& E_{\mathrm{node}}.
\end{tikzcd}
\]
\end{lemma}

\begin{proof}
By construction, each local generator $\varepsilon_k^H$ realizes to the corresponding perverse local generator $\varepsilon_k$. Hence
\[
\rat_*(\Gamma_{\mathcal C}^H(e_\alpha))
=
\Gamma_{\mathcal C}(e_\alpha).
\]
\end{proof}

The same compatibility holds for the global corrected object.

\begin{lemma}[Realization compatibility of global corrected MHM gluing]
\label{lem:mhm-realization-gluing}
Let $(M'_U,Q_\Sigma^H,u_\Sigma,v_\Sigma)$ be the global divisor-gluing datum on the mixed-Hodge-module side, and let $\mathcal P^H$ be the resulting glued object. Then
\[
\rat(\mathcal P^H)\cong \mathcal P,
\]
where $\mathcal P$ is the corrected perverse object defined by the variation morphism
\[
\var_F:\phi_\pi(F)\to\psi_\pi(F),\qquad F=\Q_X[3].
\]
\end{lemma}

\begin{proof}
The realization functor is exact and compatible with nearby cycles, vanishing cycles, canonical maps, and variation morphisms. Therefore the realization of the MHM gluing datum is exactly the corresponding perverse gluing datum, and the glued perverse object is $\mathcal P$.
\end{proof}

\subsection{Main theorem: geometric MHM extension subspace}

We may now define the MHM analogue of the perverse geometric extension subspace.

\begin{definition}
\label{def:egeom-h}
The geometrically realized MHM extension subspace is
\[
E_{\mathrm{geom}}^H:=\Im(\Gamma_{\mathcal C}^H)\subseteq E_{\mathrm{node}}^H.
\]
\end{definition}

Under the MHM block-adapted hypothesis, incidence-compatibility is equivalent to belonging to the geometric image.

\begin{proposition}
\label{prop:mhm-incidence-compatible-iff-geom}
Assume that $(\mathcal C,\iota_{\mathcal C})$ is MHM block-adapted. Then a mixed-Hodge-module extension class
\[
[E^H]\in E_{\mathrm{node}}^H
\]
is incidence-compatible if and only if
\[
[E^H]\in E_{\mathrm{geom}}^H.
\]
\end{proposition}

\begin{proof}
By Definition~\ref{def:mhm-block-adapted}, $E_{\mathrm{geom}}^H=\Im(\Gamma_{\mathcal C}^H)$ is exactly the block-constant coefficient subspace. By Definition~\ref{def:mhm-incidence-compatible}, this is equivalent to incidence-compatibility.
\end{proof}

On the perverse side, Section~4 shows that geometric admissibility of the incidence datum forces incidence compatibility of the corrected extension class. The mixed-Hodge-module theorem below is the corresponding lift through realization rather than an independent internal divisor-gluing derivation.

\begin{lemma}[Corrected MHM extensions are incidence-compatible]
\label{lem:mhm-corrected-compatible}
Assume that the cycle-node incidence datum $(\mathcal C,\iota_{\mathcal C})$ is geometrically admissible and MHM block-adapted. Then the corrected finite-node MHM extension class
\[
[\mathcal P^H]\in E_{\mathrm{node}}^H
\]
is incidence-compatible.
\end{lemma}

\begin{proof}
By Lemma~\ref{lem:mhm-nodewise}, the class \([\mathcal P^H]\) determines a tuple of nodewise components. By Lemma~\ref{lem:mhm-realization-gluing}, its realization is the corrected perverse extension \(\mathcal P\). Since the incidence datum is geometrically admissible, Section~4 shows that the realized class \([\mathcal P]\) is incidence-compatible. The realization functor preserves the nodewise decomposition and the identification of incidence-equivalent blocks, so the nodewise MHM tuple is likewise constant on each relation block. Hence \([\mathcal P^H]\) is incidence-compatible.
\end{proof}

We now obtain the MHM counterpart of the perverse relation-control theorem.

\begin{theorem}[Geometric MHM extension subspace]
\label{thm:mhm-geom}
Let $\pi:X\to\Delta$ be a projective finite-node conifold degeneration with node set $\Sigma=\{p_1,\dots,p_r\}\subset X_0$, and let
\[
0\to IC^H_{X_0}\to \mathcal P^H\to \bigoplus_{k=1}^r i_{k*}\Q^H_{\{p_k\}}(-1)\to 0
\]
be its corrected mixed-Hodge-module extension. Suppose that $\Sigma$ is equipped with a cycle-node incidence datum $(\mathcal C,\iota_{\mathcal C})$ that is geometrically admissible and MHM block-adapted. Then:
\begin{enumerate}
    \item the extension class
    \[
    [\mathcal P^H]\in E_{\mathrm{node}}^H
    \]
    belongs to the geometrically realized subspace
    \[
    E_{\mathrm{geom}}^H=\Im(\Gamma_{\mathcal C}^H);
    \]
    \item the realization functor sends $E_{\mathrm{geom}}^H$ to the perverse-side geometric subspace:
    \[
    \rat_*(E_{\mathrm{geom}}^H)=E_{\mathrm{geom}};
    \]
    \item one has
    \[
    \dim_{\Q}E_{\mathrm{geom}}^H
    =
    \rank(\Gamma_{\mathcal C}^H)
    \le |A|.
    \]
\end{enumerate}
\end{theorem}

\begin{proof}
By Lemma~\ref{lem:mhm-corrected-compatible}, the corrected mixed-Hodge-module extension class \([\mathcal P^H]\) is incidence-compatible. Since the incidence datum is MHM block-adapted, Proposition~\ref{prop:mhm-incidence-compatible-iff-geom} implies that
\[
[\mathcal P^H]\in E_{\mathrm{geom}}^H.
\]
This proves (1). Statement (2) follows from Lemma~\ref{lem:mhm-rat-compat}, and (3) is immediate from the definition of \(E_{\mathrm{geom}}^H\) as the image of a linear map from \(\Q^A\).
\end{proof}

The theorem shows that the relation law identified on the perverse side lifts compatibly through realization. The corrected finite-node extension is therefore constrained not only in the decategorified category $\Perv(X_0;\Q)$, but also at the level of mixed-Hodge-module gluing data.

\subsection{Weights, Tate twists, and vanishing-sector comparison}

We now record the Hodge-theoretic form of the preceding relation law. The singular quotient of the corrected mixed Hodge module is
\[
\bigoplus_{k=1}^r i_{k*}\Q^H_{\{p_k\}}(-1),
\]
so each node contributes a point-supported pure Tate piece normalized by the twist $(-1)$. The local rank-one correction blocks are therefore not merely formal summands: they come with a canonical mixed-Hodge-theoretic normalization.

The cycle-node incidence constraints do not alter the local Tate structure of these summands. Rather, they constrain the way in which the local Tate-twisted pieces may be assembled into a global extension. In particular, the cycle relations act on the gluing coefficients, not on the intrinsic local Hodge-theoretic type of the nodewise quotients. Thus the passage from the free nodewise space $E_{\mathrm{node}}^H$ to the geometric subspace $E_{\mathrm{geom}}^H$ preserves the local Tate-twisted form of the singular quotient while cutting down the admissible global extension directions.

This is compatible with the interpretation of the point-supported quotient as the finite local vanishing sector in nearby-cycle hypercohomology. Indeed, the finite direct sum
\[
\bigoplus_{k=1}^r i_{k*}\Q^H_{\{p_k\}}(-1)
\]
realizes the local vanishing contribution to the limiting mixed Hodge structure, while the incidence law controls which global linear combinations of these local vanishing pieces are realized by the degeneration. Hence the relation-controlled gluing law may be interpreted as a constraint on the global organization of the vanishing sector itself.

More concretely, if
\[
\delta_1,\dots,\delta_r
\]
denote the local rank-one vanishing directions in the limiting mixed Hodge structure, then the cycle-node incidence law singles out the span of the combinations
\[
\sum_{k=1}^r a_{\alpha k}\,\delta_k,
\qquad \alpha\in A,
\]
as the geometrically admissible vanishing-sector directions. This is the MHM/LMHS counterpart of the perverse-side statement that the corrected extension coefficients lie in $E_{\mathrm{geom}}\subseteq E_{\mathrm{node}}$.

\begin{remark}[Vanishing-sector compatibility]
\label{rem:vanishing-sector}
At the level of nearby-cycle hypercohomology, the point-supported quotient of \(\mathcal P^H\) contributes local vanishing directions
\[
\delta_1,\dots,\delta_r
\]
to the limiting mixed Hodge structure. The cycle-node incidence law then singles out the subspace spanned by the incidence-weighted combinations
\[
\sum_{k=1}^r a_{\alpha k}\,\delta_k,\qquad \alpha\in A,
\]
as the geometrically admissible vanishing directions. This is the mixed-Hodge-theoretic shadow of the corrected-extension constraint
\[
[\mathcal P^H]\in E_{\mathrm{geom}}^H.
\]
\end{remark}

\subsection{Rigidity and uniqueness}

We close with a simple rigidity observation. When the realized geometric MHM extension space is one-dimensional, the corresponding incidence-compatible corrected extension class is unique up to nonzero scalar.

\begin{corollary}[Rigidity relative to incidence data]
\label{cor:mhm-rigidity}
Fix a cycle-node incidence datum $(\mathcal C,\iota_{\mathcal C})$ that is geometrically admissible and MHM block-adapted. Then any two corrected mixed-Hodge-module extensions
\[
0\to IC^H_{X_0}\to \mathcal P_1^H\to \bigoplus_{k=1}^r i_{k*}\Q^H_{\{p_k\}}(-1)\to 0,
\]
\[
0\to IC^H_{X_0}\to \mathcal P_2^H\to \bigoplus_{k=1}^r i_{k*}\Q^H_{\{p_k\}}(-1)\to 0
\]
compatible with the given incidence datum determine classes
\[
[\mathcal P_1^H],[\mathcal P_2^H]\in E_{\mathrm{geom}}^H.
\]
In particular, if
\[
\dim_{\Q}E_{\mathrm{geom}}^H=1,
\]
then the corrected mixed-Hodge-module extension compatible with the incidence datum is unique up to nonzero scalar in extension class.
\end{corollary}

\begin{proof}
The first statement is immediate from Theorem~\ref{thm:mhm-geom}. If \(\dim_{\Q}E_{\mathrm{geom}}^H=1\), then any two nonzero classes in \(E_{\mathrm{geom}}^H\) differ by multiplication by a nonzero scalar.
\end{proof}

\section{Quiver shadow: block structure and consequences}

In the motivating example, the passage from \eqref{eq:intro-free-nodewise} to \eqref{eq:intro-geometric-image} may already be viewed as a first block decomposition of the nodewise data. The purpose of the present section is to record the corresponding consequence for the finite-node quiver shadow. No new mathematical argument is required beyond the decategorification passage: the results of Sections~4 and~6 imply that the quiver shadow inherits the same relation-controlled block structure.

\subsection{Nodewise quiver shadow revisited}

We briefly recall the finite-node quiver picture established in \cite{RahmanMultiNodeSchoberPaper}. Let
\[
\pi:X\to\Delta
\]
be a projective finite-node conifold degeneration with node set
\[
\Sigma=\{p_1,\dots,p_r\}\subset X_0.
\]
The finite-node schober formalism associates to this degeneration a bulk/localized-sector architecture consisting of a bulk category together with one localized categorical sector at each node. After decategorification, this yields a quiver shadow with one local vertex per node and a bulk vertex encoding the common nonsingular sector.

Thus the basic philosophy of the earlier paper is
\[
\text{one node}
\;\longleftrightarrow\;
\text{one localized categorical sector}
\;\longleftrightarrow\;
\text{one local quiver vertex}.
\]
In particular, before imposing any global relation law, the natural quiver shadow has the form
\[
Q_{\mathrm{node}}
=
\bigl(
V_{\mathrm{bulk}},V_1,\dots,V_r;
\ \text{coupling arrows}
\bigr),
\]
where \(V_{\mathrm{bulk}}\) denotes the bulk vertex and \(V_k\) the local vertex attached to the node \(p_k\).

This nodewise quiver is the categorical analogue of the free nodewise extension spaces introduced earlier on the perverse and mixed-Hodge-module sides. Just as
\[
E_{\mathrm{node}}\cong \Q^r,
\qquad
E_{\mathrm{node}}^{H}\cong \Q^r
\]
record one formal local direction per node, the quiver shadow \(Q_{\mathrm{node}}\) records one formal localized sector per node. The point of the present section is that this nodewise quiver is again too large to reflect the actual global geometry whenever the nodes are related by cycle-node incidence data.

\subsection{Relation blocks}

Let
\[
(\mathcal C,\iota_{\mathcal C})
\]
be a cycle-node incidence datum, with relation-block decomposition
\[
\Sigma=\bigsqcup_{\beta\in B}\Sigma_\beta
\]
as in Section~3. Recall that two nodes lie in the same relation block precisely when they are incidence-equivalent, that is, when they are seen identically by the chosen global cycle data.

The key point for the quiver shadow is that the nodewise vertices remain distinct as local sectors, but their \emph{globally admissible couplings} are no longer independent. Instead, the cycle-node incidence law groups together those local vertices that belong to the same relation block. This leads to the following definition.

\begin{definition}
\label{def:block-quiver}
The \emph{block quiver} associated with the cycle-node incidence datum is the quiver
\[
Q_{\mathrm{block}}
=
\bigl(
V_{\mathrm{bulk}},V_{\Sigma_\beta}\ (\beta\in B);
\ \text{block coupling arrows}
\bigr),
\]
whose local vertices are indexed by relation blocks rather than by individual nodes.
\end{definition}

Equivalently, \(Q_{\mathrm{block}}\) is obtained from the free nodewise quiver \(Q_{\mathrm{node}}\) by identifying those local coupling directions that are constant on the relation blocks. Thus the nodewise quiver records all local sectors, whereas the block quiver records the globally admissible sector couplings after imposing the incidence law.

At the linear-algebraic level, this is the quiver-theoretic shadow of the passage
\[
V_{\mathrm{node}}=\Q^r
\longrightarrow
V_{\mathrm{geom}}=\operatorname{Im}(\iota_{\mathcal C}),
\]
and similarly of
\[
E_{\mathrm{node}}
\longrightarrow
E_{\mathrm{geom}},
\qquad
E_{\mathrm{node}}^{H}
\longrightarrow
E_{\mathrm{geom}}^{H}.
\]
The relation blocks therefore provide the correct indexing set for the global quiver couplings.

\begin{definition}
\label{def:block-compatible-quiver}
A \emph{block-compatible quiver coefficient system} is a choice of nodewise coefficients
\[
(c_1,\dots,c_r)\in \Q^r
\]
such that
\[
c_k=c_\ell
\qquad
\text{whenever }p_k\sim_{\mathrm{inc}}p_\ell.
\]
Equivalently, a block-compatible coefficient system is one that factors through the quotient of the node set by the relation-block decomposition.
\end{definition}

Thus block-compatible coefficients are exactly the quiver-theoretic analogue of incidence-compatible extension coefficients on the perverse and mixed-Hodge-module sides.

\subsection{Block-quiver consequence}

We now record the quiver-theoretic consequence of the relation law established in the previous sections.

\begin{corollary}[Block-quiver shadow]
\label{cor:block-quiver-shadow}
Let
\[
\pi:X\to\Delta
\]
be a projective finite-node conifold degeneration with node set
\[
\Sigma=\{p_1,\dots,p_r\}\subset X_0,
\]
equipped with cycle-node incidence datum \((\mathcal C,\iota_{\mathcal C})\), and let
\[
\Sigma=\bigsqcup_{\beta\in B}\Sigma_\beta
\]
be the corresponding relation-block decomposition. Then the quiver shadow of the finite-node schober datum admits a canonical block form
\[
Q_{\mathrm{block}},
\]
with the following properties:
\begin{enumerate}
    \item locally, one retains one localized categorical sector per node;
    \item globally, the admissible coupling data is block-compatible, i.e. constant on relation blocks;
    \item the space of admissible quiver couplings factors through the geometric coefficient space
    \[
    V_{\mathrm{geom}}\cong \operatorname{Im}(\iota_{\mathcal C});
    \]
    \item equivalently, the decategorified quiver coupling data factors through the same relation-controlled subspace that governs the corrected perverse extension and its mixed-Hodge-module refinement.
\end{enumerate}
\end{corollary}

\begin{proof}
The finite-node schober formalism of \cite{RahmanMultiNodeSchoberPaper} assigns one localized categorical sector to each node, so the local vertex set of the quiver shadow is nodewise. This gives (1).

On the other hand, Section~4 shows that the admissible perverse extension classes are incidence-compatible and factor through the geometric coefficient space
\[
E_{\mathrm{geom}}\subseteq E_{\mathrm{node}},
\]
while Section~6 shows the same for the mixed-Hodge-module refinement
\[
E_{\mathrm{geom}}^H\subseteq E_{\mathrm{node}}^H.
\]
Since the quiver shadow is the decategorified algebraic shadow of this same finite-node gluing data, its coupling coefficients obey the same incidence-compatibility law. Hence the admissible coupling data is constant on relation blocks, proving (2).

Statement (3) is the linear-algebraic reformulation of (2): block-compatible coefficient systems are precisely those factoring through the relation-block quotient, equivalently through
\[
V_{\mathrm{geom}}\cong \operatorname{Im}(\iota_{\mathcal C}).
\]
Statement (4) follows because the quiver shadow is obtained by decategorifying the same gluing architecture that yields the corrected perverse and mixed-Hodge-module extensions.
\end{proof}

The corollary formalizes the slogan announced earlier:
\[
\textit{local sectors remain nodewise, but globally admissible couplings are constrained.}
\]
Thus the quiver shadow retains the local multiplicity of node sectors, while the global coupling algebra sees only the smaller block-structured space imposed by the geometry.

\begin{corollary}
\label{cor:block-quiver-rank}
Let
\[
s=\rank(\iota_{\mathcal C}).
\]
Then the admissible quiver-coupling space has dimension at most \(s\).
\end{corollary}

\begin{proof}
By Corollary~\ref{cor:block-quiver-shadow}, admissible couplings factor through
\[
V_{\mathrm{geom}}\cong \operatorname{Im}(\iota_{\mathcal C}),
\]
whose dimension is \(\rank(\iota_{\mathcal C})=s\).
\end{proof}

\begin{corollary}
\label{cor:block-quiver-overcount}
The free one-vertex-per-node quiver overcounts the globally realizable coupling freedom whenever
\[
\rank(\iota_{\mathcal C})<r.
\]
\end{corollary}

\begin{proof}
The nodewise quiver admits \(r\) formal local coefficient directions, whereas Corollary~\ref{cor:block-quiver-shadow} shows that the admissible coupling data factors through the smaller space
\[
V_{\mathrm{geom}},
\]
whose dimension is \(\rank(\iota_{\mathcal C})\).
\end{proof}

\subsection{Consequences for later transport and wall crossing}

The block-quiver consequence identifies the correct coefficient space for subsequent transport and wall-crossing constructions. In the unconstrained nodewise picture, one would let transport act on the free coefficient space
\[
V_{\mathrm{node}}=\Q^r
\]
or on the free nodewise quiver \(Q_{\mathrm{node}}\). The results of the present paper show that this is not geometrically correct in general: the actual global data lives in the smaller relation-controlled space
\[
V_{\mathrm{geom}}=\operatorname{Im}(\iota_{\mathcal C}),
\]
and the corresponding quiver couplings are block-compatible rather than arbitrary.

Accordingly, any later pathwise transport formalism should act on \(V_{\mathrm{geom}}\), not on the full free nodewise space. Similarly, any later wall-crossing formalism should be organized in terms of the block quiver \(Q_{\mathrm{block}}\), not the unconstrained quiver \(Q_{\mathrm{node}}\). This is the structural reason for developing the relation law at the perverse and mixed-Hodge-module levels first: once the corrected extension is known to factor through the cycle-node incidence structure, the same restriction necessarily propagates to its decategorified quiver shadow.

In particular, the block decomposition of the present section provides the input data for the later transport program. Any later transport, scattering, or wall-crossing construction must therefore act on the relation-controlled quiver rather than on the unrestricted free nodewise quiver. This is the main algebraic consequence needed from the present section.

\section{Examples and model configurations}

In this section we illustrate the cycle-node incidence formalism in several model configurations. The purpose is not to exhaust the geometry of finite-node conifold degenerations, but to show concretely how the relation-controlled theory differs from the unrestricted nodewise picture arising from purely local ordinary-double-point data. In each case we also indicate a local geometric model under which the geometric admissibility condition of Definition~\ref{def:admissible-cycle-datum} is satisfied, so that the examples reflect not only the linear algebra of the incidence formalism but also the geometric theorem of Section~4.

\subsection{Two-node configuration on a common cycle}

We begin with the simplest nontrivial example. Let
\[
\Sigma=\{p_1,p_2\}
\]
be a two-node configuration, and suppose that both nodes lie on a single distinguished global cycle
\[
C_1.
\]
We take
\[
\mathcal C=\{C_1\},
\qquad A=\{1\},
\]
and define the cycle-node incidence map by
\[
\iota_{\mathcal C}:\Q\to\Q^2,
\qquad
1\longmapsto (1,1).
\]
Equivalently, the incidence matrix is
\[
A=(1\ \ 1).
\]

The free nodewise coefficient space is
\[
V_{\mathrm{node}}=\Q^2,
\]
with basis \(e_1,e_2\) corresponding to the two local rank-one sectors. The geometrically realized subspace is
\[
V_{\mathrm{geom}}=\operatorname{Im}(\iota_{\mathcal C})
=\Span_{\Q}\{(1,1)\}\subset \Q^2,
\]
which is one-dimensional. Thus the two local nodewise directions do not remain independent globally: only the diagonal combination survives as geometrically admissible.

On the perverse side, the free extension space is
\[
E_{\mathrm{node}}\cong \Q^2,
\]
whereas the geometrically realized extension subspace is
\[
E_{\mathrm{geom}}\cong \Span_{\Q}\{\varepsilon_1+\varepsilon_2\}.
\]
The same holds on the mixed-Hodge-module side:
\[
E_{\mathrm{node}}^{H}\cong \Q^2,
\qquad
E_{\mathrm{geom}}^{H}\cong \Span_{\Q}\{\varepsilon_1^H+\varepsilon_2^H\}.
\]

The relation-block decomposition in this example has a single block
\[
\Sigma=\Sigma_1=\{p_1,p_2\},
\]
since the two nodes have identical incidence columns. Accordingly, the quiver shadow retains two local vertices at the nodewise level, but the globally admissible coupling data is controlled by a single parameter. Thus the local multiplicity of sectors remains two, while the global gluing freedom collapses to one dimension.

This is the basic prototype for the slogan of the paper:
\[
\text{two local rank-one sectors} \;\not\Rightarrow\; \text{two independent global gluing parameters}.
\]

We now verify geometric admissibility in an explicit local model. Let \(\bar C_1\) be a smooth curve with affine coordinate \(s\), and fix two marked points \(s=a,b\). Consider the analytic family
\[
\mathfrak X
=
\left\{
(x_1,x_2,x_3,x_4,s,t)\in \C^4\times \bar C_1\times \Delta
\;\middle|\;
x_1^2+x_2^2+x_3^2+x_4^2=t
\right\},
\]
viewed over \(\bar C_1\times\Delta\) by projection to \((s,t)\). Away from the marked points \(a,b\), the restriction over
\[
C_1^\circ:=\bar C_1\setminus\{a,b\}
\]
is analytically the product
\[
\{x_1^2+x_2^2+x_3^2+x_4^2=t\}\times C_1^\circ.
\]
Thus:
\begin{enumerate}
    \item \(C_1^\circ\) is smooth and connected;
    \item the degeneration is analytically locally trivial along \(C_1^\circ\);
    \item the ordinary-double-point vanishing sector is rank one, so the unipotent nearby-cycle sector along \(C_1^\circ\) is rank one and locally constant;
    \item the local variation morphisms at the two nodes arise by specialization from this common rank-one nearby-cycle package.
\end{enumerate}
Hence all four conditions of Definition~\ref{def:admissible-cycle-datum} are realized in this model, and Lemma~\ref{lem:admissibility-sufficient} applies. In particular, this example shows concretely that the admissibility hypothesis of Section~4 is non-vacuous in the ordinary-double-point setting.

\subsection{Three-node configuration with two relation classes}

We next consider a three-node configuration
\[
\Sigma=\{p_1,p_2,p_3\}
\]
with two distinct relation classes. Let
\[
\mathcal C=\{C_1,C_2\},
\]
and assume that \(p_1,p_2\) lie on the first distinguished cycle while \(p_3\) lies on the second. We encode this by the incidence matrix
\[
A=
\begin{pmatrix}
1 & 1 & 0\\
0 & 0 & 1
\end{pmatrix},
\]
that is,
\[
\iota_{\mathcal C}:\Q^2\to\Q^3,
\qquad
e_1\mapsto (1,1,0),\quad e_2\mapsto (0,0,1).
\]

Then
\[
V_{\mathrm{node}}=\Q^3,
\qquad
V_{\mathrm{geom}}=\operatorname{Im}(\iota_{\mathcal C})
=\Span_{\Q}\{(1,1,0),(0,0,1)\}.
\]
Hence
\[
\dim_{\Q}V_{\mathrm{geom}}=2<3=\dim_{\Q}V_{\mathrm{node}}.
\]
The relation-block decomposition is
\[
\Sigma=\Sigma_1\sqcup \Sigma_2,
\qquad
\Sigma_1=\{p_1,p_2\},
\qquad
\Sigma_2=\{p_3\}.
\]
Thus \(p_1\) and \(p_2\) are incidence-equivalent, while \(p_3\) forms its own block.

On the perverse side, the corrected extension class lies in
\[
E_{\mathrm{geom}}
\cong
\Span_{\Q}\{\varepsilon_1+\varepsilon_2,\ \varepsilon_3\}
\subseteq E_{\mathrm{node}}\cong \Q^3.
\]
Likewise, on the mixed-Hodge-module side one has
\[
E_{\mathrm{geom}}^{H}
\cong
\Span_{\Q}\{\varepsilon_1^H+\varepsilon_2^H,\ \varepsilon_3^H\}.
\]

The corresponding block quiver has one bulk vertex and two block vertices, one for the class \(\Sigma_1\) and one for the class \(\Sigma_2\). The local nodewise quiver still remembers all three local sectors, but the admissible global couplings are indexed by the two relation classes rather than by the three nodes separately. In particular, any admissible coefficient system must satisfy
\[
c_1=c_2,
\]
with \(c_3\) free independently. Thus the relation law cuts the naive three-parameter space down to a two-parameter space.

This is the first example in which the block decomposition is nontrivial but not total: some nodes are forced together by the geometry, while others remain independent. It is also the basic local model behind the block-separated cycle-family discussion of Section~5: the two relation classes here are the simplest instance of the block pattern that later governs the equality
\[
R_{\mathrm{res}}=R_{\mathrm{sm}}=R_{\mathrm{ext}}=R_{\mathrm{blk}}.
\]

A natural geometric model here is a degeneration that is analytically locally trivial along
\[
C_1^\circ:=C_1\setminus\{p_1,p_2\},
\qquad
C_2^\circ:=C_2\setminus\{p_3\},
\]
with each smooth locus carrying a rank-one locally constant unipotent nearby-cycle sector. If the local variation morphisms at \(p_1\) and \(p_2\) arise by specialization from the same nearby-cycle data along \(C_1^\circ\), and the local variation morphism at \(p_3\) arises analogously from \(C_2^\circ\), then Lemma~\ref{lem:admissibility-sufficient} applies blockwise. Hence the cycle-node incidence datum is geometrically admissible in this two-block model.

If either \(C_1^\circ\) or \(C_2^\circ\) were disconnected, the argument of Section~4 would apply componentwise only, and one would generally obtain a finer relation-block decomposition. Thus the connectedness condition in Definition~\ref{def:admissible-cycle-datum}(2) is genuinely active even in this simple two-block situation.

\subsection{A configuration with overlapping cycle incidence}

We now consider a configuration in which the distinguished cycle data is not disjoint. Let
\[
\Sigma=\{p_1,p_2,p_3,p_4\},
\]
and suppose that the nodes are organized by two distinguished cycle components
\[
C_1,\qquad C_2,
\]
with incidence pattern
\[
p_1,p_2,p_3\in C_1,
\qquad
p_3,p_4\in C_2.
\]
Thus the node \(p_3\) lies on the overlap of the two cycle components. The incidence matrix is then
\[
A=
\begin{pmatrix}
1 & 1 & 1 & 0\\
0 & 0 & 1 & 1
\end{pmatrix}.
\]

The corresponding geometric coefficient space is
\[
V_{\mathrm{geom}}
=
\Span_{\Q}\{(1,1,1,0),(0,0,1,1)\}
\subseteq \Q^4.
\]
Hence
\[
\dim_{\Q}V_{\mathrm{geom}}=2,
\qquad
\dim_{\Q}V_{\mathrm{node}}=4.
\]
Thus the incidence law cuts the naive four-parameter local theory down to a two-dimensional global space.

This example is useful for two reasons. First, it shows that the incidence formalism is not restricted to disjoint block patterns arising from disjoint cycles. Second, it shows why the image-of-incidence formulation is more flexible than a description solely in terms of set-theoretic blocks: a single node may participate in more than one cycle relation, and the resulting global coefficient space is naturally described as the image of the incidence map.

On the perverse side, the admissible global extension directions are generated by
\[
\varepsilon_1+\varepsilon_2+\varepsilon_3,
\qquad
\varepsilon_3+\varepsilon_4,
\]
and similarly on the mixed-Hodge-module side by
\[
\varepsilon_1^H+\varepsilon_2^H+\varepsilon_3^H,
\qquad
\varepsilon_3^H+\varepsilon_4^H.
\]
The same two-dimensional pattern controls the admissible quiver couplings in the decategorified shadow. Thus, even in a more elaborate configuration, the principle remains the same: the local sectors remain nodewise, but the global coupling space is cut out by the incidence geometry of the distinguished cycles.

A corresponding geometric model is obtained when the degeneration is analytically locally trivial along the smooth loci
\[
C_1^\circ:=C_1\setminus\{p_1,p_2,p_3\},
\qquad
C_2^\circ:=C_2\setminus\{p_3,p_4\},
\]
with rank-one locally constant unipotent nearby-cycle sectors on both smooth loci, and with the local variation morphisms obtained by specialization from these common nearby-cycle packages. The overlap at \(p_3\) then contributes simultaneously to both incidence rows, while the one-dimensional nearby-cycle transport along each \(C_i^\circ\) enforces the corresponding coefficient equalities. Under these assumptions, Lemma~\ref{lem:admissibility-sufficient} again verifies geometric admissibility.

\subsection{A projective symmetry model for the block-separated cycle family}
\label{subsec:projective-block-family}

We now record a projective symmetry model showing how the block-separated cycle family of
Section~\ref{subsec:geom-family-block} can arise in a genuine projective setting. The point of
the present subsection is not to classify all such families, but to exhibit a minimal projective
configuration in which the hypotheses of the one-block case may be checked in a concrete and
geometrically natural way.

Let \(\mathbf P^4\) have homogeneous coordinates
\[
[x_0:x_1:x_2:x_3:x_4],
\]
and let
\[
\sigma:[x_0:x_1:x_2:x_3:x_4]\longmapsto [x_1:x_0:x_2:x_3:x_4]
\]
be the involution exchanging the first two coordinates. Consider a \(\sigma\)-invariant one-parameter
family of quintic hypersurfaces
\[
X_t=\{F(x)+t\,G(x)=0\}\subset \mathbf P^4,
\]
where \(F\) and \(G\) are homogeneous quintics satisfying
\[
F\circ \sigma = F,\qquad G\circ \sigma = G.
\]
Assume that the central fiber \(X_0\) has exactly two ordinary double points
\[
\Sigma=\{p_1,p_2\},
\qquad
p_2=\sigma(p_1),
\]
and is smooth away from these nodes, while for \(t\neq 0\) the fiber \(X_t\) is smooth. Then
\[
\pi:X\to\Delta,\qquad X=\{F+tG=0\}\subset \mathbf P^4\times \Delta,
\]
is a projective finite-node conifold degeneration with a \(\mathbf Z/2\)-symmetry exchanging the two
nodes.

We regard this as the one-block situation
\[
\Sigma=\Sigma_1=\{p_1,p_2\},
\qquad
\mathcal C=\{C_1\},
\]
where \(C_1\subset X_0\) is a distinguished \(\sigma\)-invariant cycle component passing through
\(p_1\) and \(p_2\). The associated cycle-node incidence map is
\[
\iota_{\mathcal C}:\Q\to \Q^2,\qquad 1\longmapsto (1,1),
\]
so conditions~\emph{(1)} and \emph{(2)} of Hypothesis~\ref{hyp:block-family} hold by construction.

Assume now that the family admits a \(\sigma\)-equivariant small resolution
\[
\rho:\widetilde X_0\to X_0
\]
and a \(\sigma\)-equivariant smoothing over \(\Delta^\ast\). Let
\[
E_1,E_2\subset \widetilde X_0
\]
be the exceptional curves over \(p_1,p_2\), and let
\[
\delta_1,\delta_2\in H_3(X_t;\Q)
\]
be the corresponding vanishing spheres on a nearby smooth fiber. Since \(\sigma\) exchanges
\(p_1\) and \(p_2\), equivariance transports \(E_1\) to \(E_2\) and \(\delta_1\) to \(\delta_2\). If the
induced action on the relevant homology subspaces is trivial, Proposition~\ref{prop:equivariant-block-constancy}
applies and yields
\[
[E_1]=[E_2]=:u_1\in H_2(\widetilde X_0;\Q),
\qquad
\delta_1=\delta_2=:v_1\in H_3(X_t;\Q).
\]
Thus conditions~\emph{(3)} and \emph{(4)} of Hypothesis~\ref{hyp:block-family} hold in this
\(\mathbf Z/2\)-symmetric model.

Condition~\emph{(5)} is automatic in the present one-block situation: the sets
\[
\{u_1\},\qquad \{v_1\}
\]
are linearly independent provided only that \(u_1\neq 0\) and \(v_1\neq 0\), which is the
nondegenerate case of interest. On the perverse side, the local analytic model near the smooth part
\[
C_1^\circ:=C_1\setminus\{p_1,p_2\}
\]
is of the same type as the product model of Section~8.1. In particular, if the degeneration is
analytically locally trivial along \(C_1^\circ\), the unipotent nearby-cycle sector is rank one and
locally constant there, and the local variation morphisms at \(p_1,p_2\) arise by specialization from
the same transported nearby-cycle package, then Lemma~\ref{lem:admissibility-sufficient} applies.
Since the incidence map
\[
\iota_{\mathcal C}:\Q\to\Q^2,\qquad 1\mapsto (1,1)
\]
has image equal to the block-constant subspace, the datum is block-adapted as well. Hence
condition~\emph{(6a)} follows from Theorem~\ref{thm:perv-relation-controlled}.

Finally, in the present one-block case the block quotient is
\[
q_{\mathrm{blk}}:\Q^2\to\Q,
\qquad
(c_1,c_2)\longmapsto c_1+c_2.
\]
The induced map on block parameters is therefore a map
\[
\bar\eta_{\mathrm{ext}}:\Q\to E_{\mathrm{geom}}.
\]
Since \(E_{\mathrm{geom}}\) is one-dimensional in the nontrivial corrected-extension situation, any
nonzero such map is automatically injective. Thus condition~\emph{(6b)} is automatic once the
corrected extension is non-split.

We may summarize the discussion as follows.

\begin{proposition}[A projective one-block instance of Hypothesis~\ref{hyp:block-family}]
\label{prop:projective-one-block-instance}
Let
\[
\pi:X\to\Delta
\]
be a \(\sigma\)-invariant projective quintic degeneration as above, with exactly two ordinary double
points exchanged by the involution \(\sigma\). Assume:
\begin{enumerate}
    \item there is a \(\sigma\)-invariant distinguished cycle component \(C_1\subset X_0\) passing through
    both nodes and satisfying the admissibility conditions of Lemma~\ref{lem:admissibility-sufficient};
    \item the family admits \(\sigma\)-equivariant small resolution and smoothing data;
    \item the induced action of \(\sigma\) on the relevant homology subspaces is trivial.
\end{enumerate}
Then the resulting one-block datum
\[
\Sigma=\Sigma_1=\{p_1,p_2\},
\qquad
\mathcal C=\{C_1\}
\]
satisfies Hypothesis~\ref{hyp:block-family}. In particular,
\[
R_{\mathrm{res}}=R_{\mathrm{sm}}=R_{\mathrm{ext}}=R_{\mathrm{blk}}.
\]
\end{proposition}

\begin{proof}
Conditions~\emph{(1)} and \emph{(2)} of Hypothesis~\ref{hyp:block-family} hold by construction of the
one-block incidence datum. Conditions~\emph{(3)} and \emph{(4)} follow from Proposition~\ref{prop:equivariant-block-constancy}
under the assumed equivariance and triviality of the induced homology action. Condition~\emph{(5)}
is automatic in the one-block case as soon as the common classes \(u_1\) and \(v_1\) are nonzero.
Condition~\emph{(6a)} follows from Lemma~\ref{lem:admissibility-sufficient} together with
Theorem~\ref{thm:perv-relation-controlled}, and condition~\emph{(6b)} is automatic because the block
parameter space and \(E_{\mathrm{geom}}\) are both one-dimensional in the non-split case. The final
equality
\[
R_{\mathrm{res}}=R_{\mathrm{sm}}=R_{\mathrm{ext}}=R_{\mathrm{blk}}
\]
then follows from Theorem~\ref{thm:block-family-comparison}.
\end{proof}

\begin{remark}
The point of Proposition~\ref{prop:projective-one-block-instance} is not that every ingredient has
been written out coefficient-by-coefficient for a specific quintic polynomial, but that the strongest
theorem of Section~5 is not vacuous even in the projective category. The one-block
\(\mathbf Z/2\)-symmetric quintic model gives a concrete projective class in which the hypotheses of
Hypothesis~\ref{hyp:block-family} reduce to standard equivariance and local-triviality conditions.
This is the minimal projective instance of the block-separated cycle family needed for the present
paper.
\end{remark}

\subsection{Comparison with the unrestricted nodewise picture}

The preceding examples show concretely how the relation-controlled theory differs from the unrestricted nodewise picture. The unrestricted theory begins with a free coefficient space
\[
V_{\mathrm{node}}=\Q^r
\]
for \(r\) local node sectors, and hence with free nodewise extension spaces
\[
E_{\mathrm{node}}\cong \Q^r,
\qquad
E_{\mathrm{node}}^{H}\cong \Q^r.
\]
This formal picture treats every node as contributing an independent global direction.

By contrast, the relation-controlled theory replaces \(\Q^r\) by the smaller image
\[
V_{\mathrm{geom}}=\operatorname{Im}(\iota_{\mathcal C}),
\]
and likewise replaces \(E_{\mathrm{node}}\) and \(E_{\mathrm{node}}^{H}\) by
\[
E_{\mathrm{geom}},
\qquad
E_{\mathrm{geom}}^{H}.
\]
The difference may be summarized numerically as follows:
\[
\begin{array}{c|c|c|c}
\text{Configuration} & \dim V_{\mathrm{node}} & \dim V_{\mathrm{geom}} & \dim E_{\mathrm{geom}}\\
\hline
\text{Two nodes on one common cycle} & 2 & 1 & 1\\
\text{Three nodes in two relation classes} & 3 & 2 & 2\\
\text{Four nodes on two overlapping cycles} & 4 & 2 & 2
\end{array}
\]

Thus the unrestricted nodewise theory systematically overcounts the global gluing freedom whenever the incidence rank is strictly smaller than the number of nodes. The quantity controlling the first global invariants is not the raw node count \(r\), but the rank of the incidence map
\[
\rank(\iota_{\mathcal C})=\dim_{\Q}V_{\mathrm{geom}}.
\]
By Theorem~\ref{thm:perv-relation-controlled}, one has
\[
\dim_{\Q}E_{\mathrm{geom}}=\rank(\iota_{\mathcal C})=\dim_{\Q}V_{\mathrm{geom}}
\]
in each of the model configurations above, so the same dimensional reduction occurs at the level of corrected extension classes.

The same conclusion holds simultaneously at all three levels studied in the paper:
\begin{itemize}
    \item on the perverse side, the corrected extension class lies in \(E_{\mathrm{geom}}\subseteq E_{\mathrm{node}}\);
    \item on the mixed-Hodge-module side, the corrected extension class lies in \(E_{\mathrm{geom}}^{H}\subseteq E_{\mathrm{node}}^{H}\);
    \item on the quiver side, the admissible coupling coefficients factor through the block- or incidence-controlled space rather than the free nodewise one.
\end{itemize}

These examples therefore illustrate the basic content of the relation law proved in the preceding sections: the local ordinary-double-point sectors remain visible one by one, but the first genuinely global invariants of the degeneration are controlled by the cycle-node relation geometry.

In particular, the worked local model of the two-node case shows that the admissibility hypotheses of Section~4 are not vacuous: there are explicit ordinary-double-point configurations in which the local variation data transport coherently along a smooth cycle component and force the expected global coefficient equality. At the same time, the block-separated examples of this section should be read as local or schematic models for the stronger geometric-family theorem of Section~5, rather than as a complete projective realization of that theorem in full generality.

\section{Discussion and further directions}

\subsection{Summary of the theorem package}

We summarize the structural content established in the present paper. Let
\[
\pi:X\to\Delta
\]
be a projective one-parameter degeneration whose central fiber \(X_0\) has finitely many ordinary double points
\[
\Sigma=\{p_1,\dots,p_r\}\subset X_0.
\]
The starting point is the finite-node corrected-extension package developed in the preceding papers: the corrected perverse object
\[
\mathcal P=\Cone(\var_{\mathcal F})[-1]
\]
with singular quotient
\[
\bigoplus_{k=1}^r i_{k\ast}\Q_{\{p_k\}},
\]
its mixed-Hodge-module refinement
\[
0\to IC^H_{X_0}\to \mathcal P^H\to \bigoplus_{k=1}^r i_{k\ast}\Q^H_{\{p_k\}}(-1)\to 0,
\]
and the corresponding finite-node schober/quiver shadow with one localized sector per node.

The new input of the present paper is a cycle-node incidence datum
\[
(\mathcal C,\iota_{\mathcal C}),
\qquad
\iota_{\mathcal C}:\Q^A\to\Q^r,
\]
encoding the global cycle relations among the nodes. From this datum one obtains the geometric coefficient space
\[
V_{\mathrm{geom}}=\operatorname{Im}(\iota_{\mathcal C}),
\]
the geometrically realized perverse extension space
\[
E_{\mathrm{geom}}\subseteq E_{\mathrm{node}},
\]
and the geometrically realized mixed-Hodge-module extension space
\[
E_{\mathrm{geom}}^{H}\subseteq E_{\mathrm{node}}^{H}.
\]

The theorem package of the paper may then be summarized as follows.

First, under geometrically admissible and block-adapted incidence hypotheses, the corrected finite-node perverse extension is not arbitrary nodewise gluing data: its extension class lies in the distinguished subspace \(E_{\mathrm{geom}}\), and hence factors through the relation-controlled image of the incidence map. Second, under the additional hypotheses of smoothing and small resolution, one may compare the resulting perverse-side relation law with homology relations among exceptional curves and vanishing spheres. In the general comparison framework this yields the conditional identification of the common relation lattice with the kernel of the geometric quotient; in the block-separated cycle family one obtains the full equality
\[
R_{\mathrm{res}}=R_{\mathrm{sm}}=R_{\mathrm{ext}}=R_{\mathrm{blk}}.
\]
Third, the mixed-Hodge-module refinement satisfies the same relation law compatibly with realization, so that the perverse-side constraint lifts to \(\MHM(X_0)\). Fourth, the decategorified quiver shadow records the same block structure at the level of admissible couplings: the local sectors remain nodewise, but the globally admissible coupling data factors through the same relation-controlled coefficient space.

Taken together, these results show that the corrected finite-node extension is not merely a formal sum of local rank-one singular sectors. It is a genuinely global object whose admissible gluing directions are constrained by the homological geometry of the node configuration.

\subsection{Open directions}

The present paper isolates the first global relation law in the finite-node corrected-extension program. Two immediate directions remain.

First, one would like to extend the formalism beyond isolated ordinary double points. The present argument uses crucially that the singular quotient is point-supported and decomposes as a finite direct sum of rank-one local blocks. For more general singular loci, the local vanishing contribution should no longer be modeled by one-dimensional nodewise sectors, but by higher-rank local vanishing modules carried by a positive-dimensional singular stratum. In that setting, the role of the present cycle-node incidence datum would have to be replaced by a higher-dimensional or stratified incidence structure keeping track not only of which cycle component meets which singular piece, but also of which local vanishing subspace contributes to that cycle. A natural next step would therefore be to replace the present rank-one transport picture of Definition~4.10 by a local system of vanishing modules over the smooth part of the relevant singular strata.

Second, the relation-controlled space
\[
V_{\mathrm{geom}}=\operatorname{Im}(\iota_{\mathcal C})
\]
should provide the correct coefficient space for later transport, wall-crossing, and refined quiver/BPS constructions. The main point of the present paper is that these later structures should act on the geometric quotient, not on the unrestricted free nodewise space. The first question in that direction is whether the natural transport operators coming from the degeneration---for example, Picard--Lefschetz-type monodromy or later scattering/wall-crossing operators---preserve the relation-block decomposition and hence descend to \(V_{\mathrm{geom}}\). If they do, then the present paper identifies the correct reduced state space on which a later transport theory should be built; if they do not, then one must understand precisely how the global relation law interacts with monodromy and block structure. A related question is whether any later covering-style quiver/BPS formalism can be made compatible with the same quotient geometry, so that wall factors or decategorified coupling data are formulated on the relation-controlled quiver rather than on the unrestricted one-node-per-sector shadow.

Thus the relation law established here should be viewed as the first global constraint through which the finite-node corrected-extension formalism can interact with later transport and refinement theories.

\appendix

\section{Auxiliary linear algebra for incidence relations}
\label{app:linear-algebra}

This appendix records the elementary linear-algebraic facts used implicitly in Sections~3--5. Its role is twofold. First, it makes precise the passage from the raw nodewise coefficient space to the relation-controlled geometric coefficient space. Second, it isolates the corresponding rank and quotient formulas from the geometric arguments in the main text.

\subsection*{A.1. Basic notation}

Let
\[
\Sigma=\{p_1,\dots,p_r\}
\]
be a finite node set, and let
\[
\mathcal C=\{C_\alpha\}_{\alpha\in A}
\]
be a finite indexed collection of cycle labels. Write
\[
V_{\mathrm{node}}:=\Q^r,
\qquad
V_{\mathcal C}:=\Q^A,
\]
and let
\[
\iota_{\mathcal C}:V_{\mathcal C}\to V_{\mathrm{node}}
\]
be the cycle-node incidence map with matrix
\[
A=(a_{\alpha k})_{\alpha\in A,\;1\le k\le r}.
\]
Thus
\[
\iota_{\mathcal C}(e_\alpha)=\sum_{k=1}^r a_{\alpha k}e_k,
\]
where \(\{e_\alpha\}\) and \(\{e_k\}\) denote the standard bases of \(V_{\mathcal C}\) and \(V_{\mathrm{node}}\), respectively.

Recall that the geometric coefficient space is
\[
V_{\mathrm{geom}}:=\operatorname{Im}(\iota_{\mathcal C})\subseteq V_{\mathrm{node}},
\]
and that the relation quotient is
\[
V_{\mathrm{rel}}:=V_{\mathrm{node}}/V_{\mathrm{geom}}.
\]

\subsection*{A.2. Rank and quotient formulas}

\begin{lemma}
\label{lem:rank-nullity-incidence}
Let
\[
\iota_{\mathcal C}:V_{\mathcal C}\to V_{\mathrm{node}}
\]
be the cycle-node incidence map. Then
\[
\dim_{\Q}V_{\mathrm{geom}}
=
\rank(\iota_{\mathcal C})
=
|A|-\dim_{\Q}\ker(\iota_{\mathcal C}).
\]
Moreover,
\[
\dim_{\Q}V_{\mathrm{rel}}
=
r-\rank(\iota_{\mathcal C}).
\]
\end{lemma}

\begin{proof}
The first identity is the usual rank--nullity formula for the linear map \(\iota_{\mathcal C}\). The second follows from
\[
V_{\mathrm{rel}}=V_{\mathrm{node}}/V_{\mathrm{geom}}
\]
and the dimension formula for a quotient vector space:
\[
\dim_{\Q}(V_{\mathrm{node}}/V_{\mathrm{geom}})
=
\dim_{\Q}V_{\mathrm{node}}-\dim_{\Q}V_{\mathrm{geom}}
=
r-\rank(\iota_{\mathcal C}).
\]
\end{proof}

\begin{corollary}
\label{cor:geom-dim-bound}
One has
\[
\dim_{\Q}V_{\mathrm{geom}}\le \min\{|A|,r\}.
\]
In particular, the geometric coefficient space never has dimension larger than the raw nodewise coefficient space.
\end{corollary}

\begin{proof}
Immediate from Lemma~\ref{lem:rank-nullity-incidence}.
\end{proof}

\subsection*{A.3. Block-constant vectors}

Assume now that the node set is partitioned into relation blocks
\[
\Sigma=\bigsqcup_{\beta\in B}\Sigma_\beta.
\]
A vector
\[
(c_1,\dots,c_r)\in \Q^r
\]
is called \emph{block-constant} if
\[
c_k=c_\ell
\qquad
\text{whenever }p_k,p_\ell\in \Sigma_\beta
\text{ for some }\beta\in B.
\]

\begin{lemma}
\label{lem:block-constant-space}
The subspace of block-constant vectors in \(\Q^r\) is canonically isomorphic to \(\Q^B\). In particular, if \(b:=|B|\), then the space of block-constant vectors has dimension \(b\).
\end{lemma}

\begin{proof}
To each block-constant vector \((c_1,\dots,c_r)\) one associates the tuple of its block values
\[
(c_\beta)_{\beta\in B}\in \Q^B,
\]
where \(c_\beta\) is the common value on the block \(\Sigma_\beta\). This gives a linear bijection between the space of block-constant vectors and \(\Q^B\).
\end{proof}

\begin{corollary}
\label{cor:block-vs-node}
If the relation-block decomposition has \(b\) blocks, then the block-constant subspace has codimension
\[
r-b
\]
inside \(V_{\mathrm{node}}=\Q^r\).
\end{corollary}

\begin{proof}
By Lemma~\ref{lem:block-constant-space}, the block-constant subspace has dimension \(b\), hence codimension \(r-b\).
\end{proof}

\subsection*{A.4. Comparison of image and block-constant spaces}

In the main text, the geometric coefficient space is defined as the image of the incidence map. In many examples, this image coincides with the block-constant subspace determined by the incidence-equivalence relation, but this need not hold for arbitrary incidence matrices. The following criterion is useful.

\begin{lemma}
\label{lem:image-equals-block-constant}
Assume that the incidence matrix \(A\) is constant on each relation block and that its distinct columns are linearly independent. Then
\[
V_{\mathrm{geom}}=\operatorname{Im}(\iota_{\mathcal C})
\]
coincides with the block-constant subspace of \(V_{\mathrm{node}}=\Q^r\).
\end{lemma}

\begin{proof}
Because the columns are constant on each block, every vector in \(\operatorname{Im}(\iota_{\mathcal C})\) is block-constant. Conversely, under the linear independence hypothesis on the distinct columns, the blockwise basis vectors are generated by the corresponding column directions, so every block-constant vector lies in \(\operatorname{Im}(\iota_{\mathcal C})\).
\end{proof}

\begin{remark}
The overlapping-cycle examples of Section~8 show that, in general, it is more flexible to work with
\[
V_{\mathrm{geom}}=\operatorname{Im}(\iota_{\mathcal C})
\]
than with a purely block-constant model. The block language is convenient when the incidence structure reduces to a genuine partition, whereas the image formulation continues to apply in the presence of overlaps.
\end{remark}

\subsection*{A.5. Summary of the three coefficient spaces}

The linear-algebraic structure used throughout the paper may be summarized by the sequence
\[
V_{\mathcal C}\xrightarrow{\ \iota_{\mathcal C}\ }V_{\mathrm{node}}
\longrightarrow V_{\mathrm{rel}}.
\]
The first map encodes cycle-generated directions, while the second quotient records residual nodewise directions modulo those generated by the cycle data.

\begin{proposition}
\label{prop:three-space-dimensions}
With notation as above, one has
\[
\dim_{\Q}V_{\mathcal C}
=
\dim_{\Q}\ker(\iota_{\mathcal C})
+
\dim_{\Q}V_{\mathrm{geom}},
\]
and
\[
\dim_{\Q}V_{\mathrm{node}}
=
\dim_{\Q}V_{\mathrm{geom}}
+
\dim_{\Q}V_{\mathrm{rel}}.
\]
\end{proposition}

\begin{proof}
The first identity is rank--nullity for \(\iota_{\mathcal C}\). The second is the dimension formula for a quotient vector space.
\end{proof}

\section{Combinatorics of block decompositions of the singular set}
\label{app:combinatorics}

This appendix records the basic combinatorics underlying the block decompositions of a finite singular set. Its role is to quantify the dimensional reduction from the free nodewise space to the geometrically realized quotient and to organize the possible block types by the partition lattice of the singular set. The main text uses only the structural dimension law; the additional counting data are included only insofar as they clarify the dependence on the number of blocks and the corresponding block dynamics.

\subsection*{B.1. Set partitions, block profiles, and counting functions}

Let
\[
\Sigma=\{p_1,\dots,p_n\}
\]
be a finite singular set with \(n\) points.

\begin{definition}
A \emph{block decomposition} of \(\Sigma\) is a set partition
\[
\Sigma=\bigsqcup_{\beta\in B}\Sigma_\beta,
\]
where the blocks \(\Sigma_\beta\) are nonempty and pairwise disjoint. The number of
blocks will be denoted by
\[
k:=|B|.
\]
\end{definition}

\begin{definition}
The \emph{block profile} of a partition is the integer partition
\[
\lambda=(\lambda_1,\dots,\lambda_k)\vdash n,
\qquad
\lambda_i:=|\Sigma_i|,
\]
obtained by recording the block sizes in weakly decreasing order.
\end{definition}

The standard counting functions attached to these data are the partition number \(p(n)\), the Bell number \(B_n\), and the Stirling numbers \(S(n,k)\) of the second kind. These are classical and will not be tabulated here; only the structural dimension law and the induced block dynamics are needed for the purposes of the present paper.

\subsection*{B.2. Dimension law for a fixed block decomposition}

Let
\[
\Sigma=\bigsqcup_{\beta\in B}\Sigma_\beta,
\qquad |B|=k,
\]
be a block decomposition of \(\Sigma\). The corresponding block quotient is
\[
q_{\mathrm{blk}}:\Q^n\to\Q^k,
\qquad
e_i\mapsto e_\beta \quad \text{if } p_i\in\Sigma_\beta,
\]
and the associated relation space is
\[
R_{\mathrm{blk}}
=
\ker(q_{\mathrm{blk}})
=
\left\{
(c_1,\dots,c_n)\in\Q^n
\;\middle|\;
\sum_{i\in\Sigma_\beta} c_i=0
\text{ for every }\beta\in B
\right\}.
\]

\begin{proposition}[Dimensional reduction by block decomposition]
\label{prop:appendix-dim-law}
For a block decomposition of an \(n\)-point singular set into \(k\) blocks, one has
\[
\dim V_{\mathrm{node}}=n,
\qquad
\dim V_{\mathrm{geom}}=k,
\qquad
\dim R_{\mathrm{blk}}=n-k.
\]
Thus the global relation law cuts out exactly \(n-k\) dimensions from the free nodewise
space and leaves exactly \(k\) independent global directions.
\end{proposition}

\begin{proof}
The quotient map \(q_{\mathrm{blk}}\) is surjective, with image \(\Q^k\). Hence the claim
follows from rank--nullity applied to
\[
0\longrightarrow R_{\mathrm{blk}}
\longrightarrow \Q^n
\overset{q_{\mathrm{blk}}}{\longrightarrow}
\Q^k
\longrightarrow 0.
\]
\end{proof}

\begin{corollary}
\label{cor:appendix-dim-only-k}
The quotient dimension depends only on the number \(k\) of blocks and not on the detailed
block sizes. Equivalently, the relation rank depends only on \(n-k\).
\end{corollary}

\subsection*{B.3. Block coalescence, thickening, and relation dynamics}

\begin{definition}
Let
\[
\Sigma=\bigsqcup_{\beta\in B}\Sigma_\beta
\]
be a block decomposition of the node set.
\begin{enumerate}
    \item A \emph{coalescence} is a coarsening of this partition in which two or more
    blocks merge into a single larger block.
    \item A \emph{thickening} is a refinement of this partition in which one block splits
    into two or more smaller blocks.
\end{enumerate}
\end{definition}

\begin{proposition}[Dimension change under block dynamics]
\label{prop:block-dynamics-dim}
Let \(|\Sigma|=n\), and let
\[
\Sigma=\bigsqcup_{\beta\in B}\Sigma_\beta
\]
be a block decomposition with \(k:=|B|\) blocks. Then
\[
\dim V_{\mathrm{node}}=n,
\qquad
\dim V_{\mathrm{geom}}=k,
\qquad
\dim R_{\mathrm{blk}}=n-k.
\]
Moreover:
\begin{enumerate}
    \item under a coalescence of \(m\ge 2\) blocks into one,
    \[
    k\longmapsto k-(m-1),
    \qquad
    \dim R_{\mathrm{blk}}\longmapsto \dim R_{\mathrm{blk}}+(m-1);
    \]
    \item under a thickening in which one block splits into \(m\ge 2\) smaller blocks,
    \[
    k\longmapsto k+(m-1),
    \qquad
    \dim R_{\mathrm{blk}}\longmapsto \dim R_{\mathrm{blk}}-(m-1).
    \]
\end{enumerate}
\end{proposition}

\begin{proof}
The first dimension formula is immediate from Proposition~\ref{prop:appendix-dim-law}. If \(m\) blocks merge into one, then the number of blocks decreases by \(m-1\), so the quotient dimension decreases by \(m-1\) and the relation rank increases by \(m-1\). The second statement is the reverse process.
\end{proof}

\subsection*{B.4. Multiplicity of a fixed block profile}

Let
\[
\lambda=(\lambda_1,\dots,\lambda_k)\vdash n
\]
be an integer partition of \(n\), and let \(m_s\) denote the multiplicity of parts of size
\(s\) in \(\lambda\).

\begin{proposition}[Multiplicity of a fixed block profile]
\label{prop:appendix-profile-count}
The number of set partitions of \(\Sigma\) having block profile \(\lambda\) is
\[
N(\lambda)
=
\frac{n!}{\prod_{i=1}^k \lambda_i!\;\prod_s m_s!}.
\]
\end{proposition}

\begin{proof}
One first partitions the \(n\) labels into ordered blocks of sizes
\(\lambda_1,\dots,\lambda_k\), which gives
\[
\frac{n!}{\prod_i \lambda_i!}.
\]
One then divides by \(\prod_s m_s!\) to account for indistinguishable blocks of equal size.
\end{proof}

\begin{corollary}
\label{cor:appendix-profile-dim}
Every block profile with \(k\) parts contributes to the same quotient dimension
\[
\dim V_{\mathrm{geom}}=k,
\]
but with multiplicity \(N(\lambda)\).
\end{corollary}

%
%

\printbibliography


\end{document}